\documentclass[11pt,eqno]{amsart}
\usepackage{}%zhou2018.10.7
\usepackage{amsmath,amssymb,txfonts}
\usepackage{amssymb}
\usepackage{amsxtra}
\usepackage{amsmath}
\usepackage{txfonts}
 \usepackage{color}

\usepackage[cp1252]{inputenc}

\usepackage{mathrsfs}
\usepackage{graphicx}

\usepackage[active]{srcltx}
\usepackage{color}

\allowdisplaybreaks

\def\rr{{\mathbb R}}
\def\rn{{{\rr}^n}}

\def\zz{{\mathbb Z}}
\def\nn{{\mathbb N}}

\def\fz{\infty}
\def\az{\alpha}
\def\dist{{\mathop\mathrm{\,dist\,}}}
\def\loc{{\mathop\mathrm{\,loc\,}}}

\def\kz{\kappa}

\def\gz{{\gamma}}

\def\wz{\widetilde}

\def\diam{{\mathop\mathrm{\,diam\,}}}

\def\r{\right}
\def\lf{\left}

\newtheorem{thm}{Theorem}[section]
\newtheorem{lem}[thm]{Lemma}

\newtheorem{rem}[thm]{Remark}
\newtheorem{cor}[thm]{Corollary}

\numberwithin{equation}{section}

\begin{document}
\arraycolsep=1pt

\title [ Orlicz-Besov imbedding  and globally $n$-regular domains]
{ Orlicz-Besov imbedding  and globally $n$-regular domains}

\author{Hongyan Sun }

\address {School of Sciences, China University of Geosciences, Beijing, 100083, P.R. China}
                    \email{sun$\_$hy@cugb.edu.cn}

%\date{  }
\maketitle

\begin{center}
\begin{minipage}{13.5cm}\small{\noindent{\bf Abstract}\quad
   Denote by $ {\bf\dot B}^{\alpha,\phi}(\Omega)$ the Orlicz-Besov space,
   where $\alpha\in\rr$, $\phi$ is a Young  function and $\Omega\subset\rn$ is a domain.
For $\alpha\in(-n,0)$ and optimal $\phi$, in this paper  we    characterize domains supporting the imbedding ${\bf\dot B}^{\alpha,\phi}(\Omega)$ into $ L^{n/|\alpha|}(\Omega)$ via globally $n$-regular domains.
 This extends the known  characterizations for
  domains supporting the Besov  imbedding ${\bf\dot B} ^s_{pp}(\Omega)$ into $ L^{np/(n-sp)}(\Omega)$ with $s\in(0,1)$ and $1\le p<n/s$. The proof of the imbedding ${\bf\dot B}^{\alpha,\phi}(\Omega)\to L^{n/|\alpha|}(\Omega)$ in
     globally $n$-regular domains $\Omega$ relies on a geometric inequality involving $\phi$ and $\Omega$ , which extends a known geometric inequality of Caffarelli et al.
}
\end{minipage}
\end{center}

%\tableofcontents
\section{Introduction\label{s1}}

Suppose that $\phi $  is a Young function in $[0,\fz)$,
that is,  \begin{equation}
\label{young}\mbox{$\phi\in C([0,\fz))$ is  convex and satisfies
 $\phi(0)=0$, $\phi(t)>0$  for $t>0$ and $\lim  _{t \to \infty} \phi (t)= \infty$.}\end{equation}
Let $n\ge 2$ and   $\Omega\subset\rn$ be a  domain.
 For $\alpha\in\rr$,   the homogenous  Orlicz-Besov space  ${\bf\dot B} ^{\alpha,\phi}(\Omega)$ is defined as the  space of
  all measurable functions $u $ in $\Omega$  whose (semi-)norms
  \begin{eqnarray*}
\|u\|_{{\bf\dot B} ^{\alpha,\phi} (\Omega)} := \inf \left\{\lambda > 0 : \int_{\Omega} \int_\Omega  \phi \left(\frac{|u(x)-u(y)|}{\lambda |x-y|^{\alpha}}\right) \frac{dxdy}{|x-y|^{2n}}\leq 1 \right\}
\end{eqnarray*}
are finite.  Modulo constant functions,
${\bf\dot B} ^{\alpha,\phi}(\Omega)$   is a  Banach space.
 Recall that the space ${\bf\dot B} ^{0,\phi }  $ in metric spaces was introduced by Piaggio \cite{c}
in the study of  geometric group theory, and the  extension and imbedding properties of ${\bf\dot B} ^{0,\phi } (\Omega)$ were considered by Liang-Zhou \cite{lz}.

The Orlicz-Besov spaces extend the Besov (or fractional Sobolev) spaces. For $s>0$ and $p\ge1$, denote by  ${\bf\dot B} ^s_{pp}(\Omega)$
 the  Besov spaces (also called fractional Sobolev space) equipped with the (semi)-norms
\begin{equation}\label{bw}
 \|u\|_{{\bf\dot B} ^s_{pp}(\Omega)}:=   \lf(\int_{\Omega} \int_{\Omega}  \frac{|u(x)-u(y)|^p}{|x-y|^{ n+s p}} \,dxdy\r)^{1/p}.
\end{equation}
Letting $\alpha=s-n/p$ and $\phi(t)=t^p$,
we always have  $\|u\|_{{\bf\dot B} ^{\alpha,\phi}(\Omega)}=
\|u\|_{{\bf\dot B} ^{s}_{pp}(\Omega)}$  for all possible $u$, and hence ${\bf\dot B} ^{\alpha,\phi}(\Omega)={\bf\dot B} ^{s}_{pp}(\Omega)$.
Thus,   Orlicz-Besov spaces include the   Besov spaces as special examples.

 Note that when $s\in(0,1)$ and $p\ge1$,
 ${\bf\dot B} ^{s}_{pp}(\Omega)$  is non-trivial, indeed, $C_c^1(\Omega)\subset {\bf\dot B} ^{s}_{pp}(\Omega)$.
But  when $s\ge1$ and $p\ge1$, thanks to the Poincar\'e inequality,
 ${\bf\dot B} ^{s}_{pp}(\Omega)$ is trivial, that is,
 contain at most  constant functions; see for example \cite[Theorems 4.1\&4.2]{gkz13}.
In general, to guarantee $C_c^1(\Omega)\subset {\bf\dot B} ^{\alpha,\phi} (\Omega)$ and hence the  non-triviality of  ${\bf\dot B} ^{\alpha,\phi} (\Omega)$, we assume that
$\phi $  be a Young function   satisfying
 \begin{equation}\label{delta0} \underline\Lambda_\phi(\alpha):=\sup_{x>0}
\int_0^1\frac{\phi( t^{1-\alpha} x)}{ \phi(x)}\frac{dt}{t^{n+1} }<\infty
\end{equation}
 and
 \begin{equation}\label{deltafz}\overline\Lambda_\phi(\alpha) := \sup_{x>0}\int_1^{\infty} \frac{\phi( t^{-\alpha} x)}{ \phi(x)} \frac{dt}{t^{n+1} }<\infty;\end{equation}
 see Lemma \ref{l2.2}.
 For the  optimality of \eqref{delta0} and \eqref{deltafz} we refer to Remark \ref{r1.4} and Remark \ref{r2.4}.
 Under \eqref{delta0} and \eqref{deltafz}, the interesting range of $\alpha$ is $(-n,1)$.
 Indeed, by the convexity of $\phi$, when $\alpha\ge1$, we always have
$$ \underline\Lambda_\phi(\alpha)\ge
\int_0^1 t^{1-\alpha}\frac{dt}{t^{n+1} }=\infty,
$$
that is, \eqref{delta0} fails.
Moreover, \eqref{deltafz} always holds when $\alpha\ge0$, but  when $\alpha\le -n$, by the convexity of $\phi$  and letting $0<c\in \partial\phi(1)$, we always have
    $$  \overline\Lambda_\phi(\alpha)\ge \sup_{x>0}\int_1^\fz \frac{\phi(1)+c(t^{-\alpha}x-1)}{\phi(x)}\frac{dt}{t^{n+1}}=\sup_{x>0}
    \left[\frac{\phi(1)-c  }{n\phi(x)}  + \frac{c x}{\phi(x)}\int_1^\fz \frac{dt}{t^{n+\alpha+1}}\right]=
    \fz,$$
    that is,  \eqref{deltafz} fails. Young functions satisfies \eqref{delta0} and \eqref{deltafz} include
    $t^{n/(1-\alpha)}[\ln(1+t)]^\gz$ with $\gz>1$ and $\alpha\in(-n+1,0)$,
    $t^p[\ln(1+t)]^\gz$ with $ p\in(n/(1-\alpha),  n/|\alpha|)\cap[1,n/|\alpha|)$ and $\gz\ge1$ or $\gz=0$, and also their convex combinations.

In this paper, we are interesting in  $\alpha\in(-n,0)$. In this case, the homogeneity of ${\bf\dot B} ^{\alpha,\phi }(\rn)$  is the same as
that of the Lebesgue space $L^{n/|\alpha|}(\rn) $, that is,
$$ \mbox{$\|u(r\cdot)\|_{{\bf\dot B} ^{\alpha,\phi } (\rn)}= r^{-\alpha} \|u \|_{{\bf\dot B} ^{\alpha,\phi } (\rn)}$
and   $\|u(r\cdot)\|_{L ^{n/|\alpha| } (\rn)}= r^{-\alpha} \|u \|_{L ^{n/|\alpha| } (\rn)}$ for any $r>0$ and function  $u$. }$$
 Viewing the theory for Sobolev spaces and Besov spaces, this enable us to establish the following imbeddings  of
$ {\bf\dot B} ^{\alpha,\phi } (\rn)$ into $L ^{n/|\alpha| } (\rn)$, and $ {\bf\dot B} ^{\alpha,\phi } (B)$ into $L ^{n/|\alpha| } (B)$ for all balls $B\subset\rn$.
\begin{thm}\label{mainthm0}
Let $\alpha\in (-n,0)$  and $\phi $  be a Young function   satisfying \eqref{delta0} and \eqref{deltafz}.
 %\begin{equation}\label{delta0} \underline\Lambda_\phi(\alpha):=\sup_{x>0}
%\int_0^1\frac{\phi( t^{1-\alpha} x)}{ \phi(x)}\frac{dt}{t^{n+1} }<\infty.
%\end{equation}
% and
% \begin{equation}\label{deltafz}\overline\Lambda_\phi(\alpha) := \sup_{x>0}\int_1^{\infty} \frac{\phi( t^{-\alpha} x)}{ \phi(x)} \frac{dt}{t^{n+1} }<\infty.\end{equation}
 Then there exists a constant $C\ge 1$ depending only on $n$, $\alpha$ and $\phi$ such that
 $$
    \|u-u_B\|_{L^{n/|\alpha|}(B)}\le C\|u\|_{{\bf\dot B}^{\alpha,\phi}(B)}\quad
   \forall\ \mbox{balls $B$ and } u\in  {\bf\dot B}^{\alpha,\phi}(B)$$
  and
 $$
\inf_{c\in\rr}\|u-c\|_{L^{n/|\alpha|}(\rn)}\le C\|u\|_{{\bf\dot B}^{\alpha,\phi}(\rn)}\quad \forall\ u\in {\bf\dot  B}^{\alpha,\phi}(\rn).
$$
 \end{thm}
 For any  $s\in(0,1)$ and $1< p<n/s$, let $\alpha=s-n/p\in(-n,0)$. By ${\bf\dot B}^s_{pp}(\Omega)=  {\bf\dot B}^{\alpha,\phi}(\Omega)$ for all $\Omega\subset\rn$
 and $n/|\alpha|=np/(n-ps)$,
Theorem \ref{mainthm0} is exactly the following well-known   imbedding   of Besov spaces ${\bf\dot B}^s_{pp} $: % into $L^{np/(n-ps)}(\rn)$:
there exists a constant $C>0$ depending only on $n,s,p$ such that
\begin{equation}\label{ebimb}
    \|u-u_B\|_{L^{np/(n-sp)}(B)}\le C\|u\|_{{\bf\dot B}^s_{pp}(B)}\quad
   \forall\ \mbox{balls $B$ and } u\in  {\bf\dot B}^s_{pp}(B)
   \end{equation}
  and
\begin{equation}\label{ebimrn}
\inf_{c\in\rr}\|u-c\|_{L^{np/(n-ps)}(\rn)}\le C\|u\|_{{\bf\dot B}^s_{pp}(\rn)}\quad \forall\ u\in {\bf\dot B}^s_{pp}(\rn)
   \end{equation}
see for example \cite{a75,gt,jw84,p76}.
% in the literature (see for example \cite{a75,gt}).

Moreover, we   characterize  all domains $\Omega \subsetneq\rn$ supporting the imbedding of
$ {\bf\dot B} ^{\alpha,\phi } (\Omega)$ into $L ^{n/|\alpha| } (\Omega)$ via  globally $n$-regular domains.
Here, a domain $\Omega\subset\rn$ is    globally  $n$-regular  if  there exists a constant $ \theta\in(0,1) $ such that
$$|B(x,r)\cap\Omega|\ge \theta r^n\quad\forall\ x\in\Omega\ {\rm and}\ 0<r<2\diam\Omega.$$
Recall that in the literature
a domain $\Omega\subset\rn$ is  { \it  $n$-regular} (or satisfies the {\it measure density property}) if  there exists a constant $\theta\in(0,1)$ such that
 $$|B(x,r)\cap\Omega|\ge \theta r^n\quad\forall\ x\in\Omega\ {\rm and}\ 0<r<1,$$
 see \cite{jw78,jw84,s06,hkt08,z14} and references therein. Obviously, a bounded domain is  $n$-regular  if and only if it is globally
    $n$-regular. A unbounded globally
   $n$-regular domain is always  $n$-regular  but the converse is not true in general.

 \begin{thm}\label{mainthm}
Let $\alpha\in (-n,0)$  and $\phi $  be a Young  function    satisfying
\eqref{delta0} and \eqref{deltafz}.
 Then, a   domain $\Omega\subset\rn$ is    globally $n$-regular   if and only if %it supports the following ${\bf\dot B}^{\alpha,\phi}$-imbedding:
    there exists a constant $C\ge 1$ such that \begin{eqnarray}\label{bdah}
 \|u-u_\Omega\|_{L^{n/|\alpha|}(\Omega)}\le C\|u\|_{{\bf\dot B}^{\alpha,\phi}(\Omega)}\quad \forall u\in {\bf\dot  B}^{\alpha,\phi}(\Omega)\quad {\rm\  when \ \diam\Omega<\fz},
\end{eqnarray}
 and  \begin{eqnarray}\label{ubdah}
 \|u \|_{L^{n/|\alpha|}(\Omega)}\le C\|u\|_{{\bf\dot B}^{\alpha,\phi}(\Omega)}\quad \forall u\in {\bf\dot  B}^{\alpha,\phi}(\Omega) \  {\rm having\ bounded\ supports} \quad {\rm when\ \diam\Omega=\fz}.
\end{eqnarray}
 \end{thm}

 As a consequence of Theorem \ref{mainthm}, we have the following results for inhomogeneous Orlicz-Besov spaces %${\bf B}^{\alpha,\phi}(\Omega)$
 ${\bf B}^{\alpha,\phi}(\Omega):= L^\phi(\Omega)\cap {\bf\dot B}^{\alpha,\phi}(\Omega)$ which are equipped with the norms
  $\|u\|_{{\bf  B}^{\alpha,\phi}(\Omega)}=\|u\|_{L^\phi(\Omega)}+ \|u\|_{ {\bf\dot B}^{\alpha,\phi}(\Omega)}.$
  Recall that $L^\phi(\Omega) $ is the Orlicz space, that is, the collection of measurable functions $u$ in $\Omega$ with
  $$\|u\|_{L^\phi(\Omega)}:=\inf\left\{\lambda>0: \int_{\Omega}\phi\left(\frac{|u(x)|}\lambda\right)\,dx\le1\right\}<\fz.$$

 %A domain $\Omega\subset\rn$ is called local  Ahlfors $n$-regular  if there exists a constant  there exist  constants $C\ge1$ and $r_0\in (0,\diam\Omega)$ such that
%$$|B(x,r)\cap\Omega|\ge Cr^n\quad\forall\ x\in\Omega\ {\rm and}\ 0<r<r_0.$$
%
% Note that Alhfors n-regular domain must be local  Ahlfors $n$-regular.
% Conversely, if a bounded local  Ahlfors $n$-regular is  Ahlfors $n$-regular.
% But there exist  unbounded domains which are  local  Ahlfors $n$-regular
% but not necessarily Ahlfors $n$-regular, say $\Omega=(0,1)\times\rn$.
 \begin{cor}\label{c1.2}
 Let $\alpha\in (-n,0)$  and $\phi$ be a Young  function  satisfying \eqref{delta0} and \eqref{deltafz}.
A bounded  domain $\Omega\subset\rn$ is    $n$-regular   if and only if %it supports the following $  {\bf  B}^{\alpha,\phi}$-imbedding:
    there exists a constant $C\ge 1$ such that \begin{eqnarray}\label{bdahnon}
 \|u \|_{L^{n/|\alpha|}(\Omega)}\le C\|u\|_{  {\bf B}^{\alpha,\phi}(\Omega)}\quad \forall u\in   {\bf B}^{\alpha,\phi}(\Omega).
\end{eqnarray}
%(ii) If $\Omega\subset\rn$ is a unbounded globally $n$-regular domain,  then \eqref{bdahnon} holds.  %a%lways supports the following $  {\bf B}^{\alpha,\phi}$-imbedding:
%then there exists a constant $C\ge 1$ such that
%  \begin{eqnarray}\label{ubdahnon}
% \|u \|_{L^{n/|\alpha|}(\Omega)}\le C\|u\|_{ {\bf B}^{\alpha,\phi}(\Omega)}\quad \forall u\in  {\bf B}^{\alpha,\phi}(\Omega).
%\end{eqnarray}
 \end{cor}

 For any  $s\in(0,1)$ and $1< p<n/s$,  
 by considering extension of ${\bf B}^s_{pp}(\Omega)$ fucntions as in 
 \cite{jw78,jw84,s06,s07,z14}, we know that
 % $n$-regular domains  characterize for   all domains
%supporting the imbedding of ${\bf  B}^s_{pp}(\Omega)$ into $L^{np/(n-ps)}(\Omega)$.
%
%Let $0<s<1$ and $0<p<n/s$. It is known that
a domain $\Omega\subset\rn$ is    $n $-regular  if and only if %it supports  the Besov $ {\bf  B} ^s_{pp}$-imbedding:
there exists a constant $C>0$   such that \begin{equation}\label{ps4}
 \|u \|_{L^{np/(n-ps)}(\Omega)}\le C\|u\|_{ {\bf   B} ^s_{pp}(\Omega)}  \quad\forall u\in {\bf  B} ^s_{pp}(\Omega). %C\left(\int_\Omega  \int_ \Omega  \frac{|u(x)-u(y)|^p}{|x-y|^{n+s p}}\,dx\,dy\right)^{1/p} .
\end{equation}
%see \cite{jw78,s06,s07,z14}.
%From this it follows that
%for $0<s<1$ and $p>0$,
By considering the extension   of ${\bf \dot B}^s_{pp}(\Omega)$ functions
similarly, one also would prove that
a  bounded domain $\Omega\subset\rn$ is globally $n $-regular  if and only if %it supports  the the Besov ${\bf\dot B} ^s_{pp}$-imbedding:
there exists a constant $C>0$   such that
\begin{equation}\label{ps41}\|u- u_\Omega\|_{L^{np/(n-ps)}(\Omega)}\le C\|u\|_{{\bf\dot B} ^s_{pp}(\Omega)}  \quad\forall u\in {\bf\dot B} ^s_{pp}(\Omega).
\end{equation}
%This also implies that    globally $n$-regular domains characterize all bounded domains supporting
%imbedding of ${\bf  \dot B}^s_{pp}(\Omega)$ into $L^{np/(n-ps)}(\Omega)$.
Letting $\alpha=s-n/p\in(-n,0)$ and $\phi(t)=t^p$, since
${\bf  B}^s_{pp}(\Omega)=  {\bf  B}^{\alpha,\phi}(\Omega)$ and ${\bf\dot B}^s_{pp}(\Omega)=  {\bf\dot B}^{\alpha,\phi}(\Omega)$, we know that
Corollary \ref{c1.2} and Theorem \ref{mainthm}   generalize  both of \eqref{ps4} and \eqref{ps41} when $\Omega$ is bounded domain. We refer to \cite{h96,h03,y03,hkt08,hkt08b,hit} for the characterization   of   domains
supporting the Hajlasz   Sobolev (Besov) imbedding into Lebesgue spaces via $n$-regular domains.

 \begin{rem}\label{r1.4}\rm
We remark that the assumption \eqref{deltafz} in Theorem 1.2 and Corollary \ref{c1.2} is   optimal for bounded domains in the following sense.
 Given any $\alpha\in(-n,0)$,
  the Young function $t^p$ with $p\ge1$ satisfies the assumption \eqref{deltafz} if and only if $p< -n/\alpha$.
 In the critical case   $\phi_0(x)=t^{ n/|\alpha|}$ with $\alpha\in(-n,0)$,
   \eqref{bdah} holds for any bounded domain $\Omega\subset \rn$.
   Indeed,  for any $u\in {\bf\dot B} ^{\alpha,\phi} (\Omega)$, by the H\"older inequality  we have
\begin{align*}
\|u-u_{B}\|_{L^{n/|\alpha|}} %&\le \frac{1}{|\Omega|}\left\{\int_{\Omega}\left[\int_{\Omega}|u(x)-u(y)|dy\right]
%^{n/|\alpha|}dx\right\}^{|\alpha|/n} \\
&\le|B|^{\alpha/n}
\left\{\int_{\Omega}\int_{\Omega}|u(x)-u(y)|^{n/|\alpha|}dxdy\right\}^{|\alpha|/n}\\
&\le \frac{|B|^{\alpha/n}}{(  \diam \Omega )^{\alpha }}\left\{
\int_{\Omega}\int_{\Omega}\phi_0\left(\frac{|u(x)-u(y)|}{|x-y|^{\alpha}}\right)
\frac{dxdy}{|x-y|^{2n}}\right\}^{|\alpha|/n},
\end{align*}
where $B$ is a fixed ball so that $2B\subset\Omega$.
Therefore, $$\int_{\Omega}\int_{\Omega}\phi_0
\left(\frac{|u(x)-u(y)|}{\frac{( \diam \Omega )^{\alpha }}{|B|^{\alpha/n}}\|u-u_{B}\|_{L^{n/|\alpha|}}|x-y|^{ \alpha}}\right)\frac{dxdy}{|x-y|^{2n}} \ge 1,$$
 which yields that
 $$ \|u-u_B\|_{L^{n/|\alpha|}(\Omega)}\le \frac{( \diam \Omega )^{\alpha}}{|B|^{\alpha/n}}\|u\|_{{\bf\dot B} ^{\alpha,\phi} (\Omega)} $$
 as desired.
\end{rem}
To prove   Theorems \ref{mainthm0}\&\ref{mainthm}, we establish the following  geometric inequality,
which extend a geometric inequality of Caffarelli-Valdinoci \cite{cv} and Savin-Valdinoci \cite{sv13} not only to general Young functions but also to globally $n$-regular domains; see Remark 3.3  (i).
%
%. Recall that when $\phi(t)=t^p$ and $\Omega=\rn$, Lemma \ref{l2.7x} was
% proved by \cite[Corollary 25  ]{cv} and   \cite[Lemma A.1]{sv13}; see Remark 3.3  for details.
\begin{lem}\label{l2.7x} Suppose that $\Omega\subset\rn$ is a globally $n$-regular domain.
Let $\alpha\in (-n,0)$  and $\phi : [0,\fz) \to [0,\infty)$  be a Young function satisfying \eqref{delta0} and \eqref{deltafz}.
 There exists  constants $C_1, C_2>0$ depending on $n,\alpha,\phi,\Omega$ such that
$$
\int_{ \Omega\setminus E  }\frac{\phi(t|x-y|^{-\alpha})}{|x-y|^{2n}}dy \ge C_1 \frac{1}{|E|} \frac{|\Omega\setminus E|}{| \Omega|} \phi(C_2t|E|^{|\alpha|/n} )
$$
 whenever $t>0$,    $x\in  \Omega$ and $E\subset  \Omega$ with $0<|E|<\fz$.
Here, if $|\Omega|=\fz$, we let $\frac{|\Omega\setminus E|}{| \Omega|}=1$.
 \end{lem}

In Section 3, applying  Lemma \ref{l2.7x},  using median values  and
improving some argument of Di Nezza et al \cite{dpv}
we prove in a direct way that  globally $n$-regular domains supporting ${\bf\dot B}^{\alpha,\phi}$-imbedding
  \eqref{bdah} or \eqref{ubdah}.  By approximating $\rn$ by balls $B(0,R)$, which have uniform  globally $n$-regular constants, we also derive Theorem 1.1.
Conversely, by  precise estimates of ${\bf\dot B}^{\alpha,\phi}$-norms  of
certain cut-off functions  as in Section 2, borrowing the ideas from Haj\l asz et al \cite{hkt08} and Zhou \cite{z14} we prove  that
domains supporting ${\bf\dot B}^{\alpha,\phi}$-imbeddings
as in \eqref{bdah} or \eqref{ubdah} are globally $n$-regular. Thus Theorem 1.2 holds.
%Note that \eqref{bdah} also hold when $\Omega$ is unbounded globally $n$-regular domain, but
% the direct approach here fails to  give it. Instead, in a forth-coming paper  we will prove it via  considering Orlicz-Besov extension.
We remark that this  gives a direct and new proof for \eqref{ebimb} and \eqref{ebimrn},
and also for \eqref{ps41} and  \eqref{ps4} in bounded $n$-regular domains; see Remark 3.3  (ii).

%of Theorem \ref{mainthm0} gives a direct and new proof for \eqref{ebimb} and \eqref{ebimrn}, see Remark 3.3  (ii). The proofs of Theorem \ref{mainthm} and Corollary \ref{c1.2} give a direct and new proof for \eqref{ps41} and hence \eqref{ps4} for bounded $n$-regular domains;

 \begin{rem}\rm
Under the assumption of Theorem \ref{mainthm},
when $\Omega$ is a unbound domain, we also have the following  characterization:
$\Omega$ is globally $n$-regular if and only if   there exists constant $C>0$ such that
 \begin{eqnarray}\label{bubdah}
 \inf_{c\in\rr}\|u-c\|_{L^{n/|\alpha|}(\Omega)}\le C\|u\|_{{\bf\dot B}^{\alpha,\phi}(\Omega)}\quad \forall u\in {\bf\dot  B}^{\alpha,\phi}(\Omega).
\end{eqnarray}
which is better than \eqref{ubdah}.
Moreover, Corollary \ref{c1.2} also holds when $\Omega$ is unbounded.  However, the direct approach above fails to prove them; see Remark 3.3 for a reason.  Instead, in a forth-coming paper \cite{sun2}, for full range $\alpha\in(-n,1)$ and Young functions $\phi$ satisfying  \eqref{delta0} and \eqref{deltafz}, we   consider the extension  of
Orlicz-Besov spaces,  then together with Theorems \ref{mainthm0} \&\ref{mainthm} and Corollary \ref{c1.2} above,  fully characterize     Orlicz-Besov imbedding domains.
\end{rem}

Finally, we make some convention on the notations or notion used in this paper.
Throughout the paper,  $C$ denotes a positive constant, which depend  only on  $n,\alpha, \phi,\Omega$ but whose value might be change from line to line.   We write $A\lesssim(\gtrsim) B$ if there exists a constant $C>0$ such that $A\le(\ge) CB$.   We use $\fint_B f(x)dx$ to denote the average $\frac{1}{|B|}\int_B f(x)dx$ of a function $f$ in a  set $B$ with positive measure.
%For every $x,y\in\Omega$, we use $|x-y|$ denote the distance between $x$ and $y$; For subset $V\subset\mathbb{R}$, we use $|V|$ denote the volume of $V$.
For any $x\in\Omega$ and $A\subset \mathbb{R}$, $\dist(x,A)$ denote the distance from $x$ to $A$, $\diam A$ denote the diameter of the set $A$.

\section{Some basic properties}

We list several basic properties  of Orlicz-Besov spaces in this section. Let us begin with two simple properties of Young  functions.  Note that   Young  functions are always strictly increasing.
\begin{lem} Let $\alpha\in (-n,0)$  and $\phi$  be a Young  function  .

(i) If $\phi$  satisfies \eqref{delta0}, then
\begin{equation}\label{delta00}  \phi ( xs ) \le 2^{2n}{\underline\Lambda_\phi(\alpha)}  {\phi(2^{1-\alpha}x)}s^{n/(1-\alpha)}\quad\forall 0<s\le 1, x>0.  \end{equation}

 (ii) If $\phi$  satisfies \eqref{deltafz}, then
  $\phi ( x s^{-\alpha})s^{-n }\to0$  as $s\to\infty$ for any $x>0$ and
\begin{equation}\label{deltafzz}  \phi ( xs ) \le 2^{3n}{\overline\Lambda_\phi(\alpha)}  {\phi(x)}s^{-n/\alpha}\quad\forall s\ge 1, x>0.  \end{equation}
\end{lem}

\begin{proof}
(i) Since $\phi$ is increasing, by \eqref{delta0}    we have
 \begin{align*} \sup_{s\in(0,1]}\frac{\phi( xs^{1-\alpha})}{\phi( x2^{1-\alpha} )s^n} &\le \sup_{j\ge0}\sup_{s\in(2^{-j-1},2^{-j}]}\frac{\phi( xs^{1-\alpha})}{\phi( x2^{1-\alpha} )s^n} \le 2^{3n} \sup_{j\ge0}\int_{2^{-j}}^{2^{-j+1}} \frac{\phi( xs^{1-\alpha})}{\phi( x2^{1-\alpha} ) }\frac{ds}{s^{ n+1}} \\
 &\le 2^{2n} \sup_{j\ge0}\int_{2^{-j-1}}^{2^{-j }} \frac{\phi( x2^{\alpha+1}s^{1-\alpha})}{\phi( x2^{1-\alpha} ) }\frac{ds}{s^{ n+1}} \le
 2^{2n} \int_0^1 \frac{\phi( x2^{1-\alpha}s^{1-\alpha})}{\phi( x2^{1-\alpha} ) }\frac{ds}{s^{ n+1}}\le 2^{2n}
 {\underline\Lambda_\phi(\alpha)}  \quad\forall x>0,
 \end{align*}
 which gives \eqref{delta00}.

(ii)
 Similarly, for all $x>0$ and $j\ge0$ we have
 \begin{align*}
  \sup_{s\in[2^{j},2^{j+1})}
  \frac{\phi ( xs^{-\alpha} )}{s^{ n}\phi(x)}  \le 2^{2n+1 }  \int_{2^{j+1}}^{2^{j+2}}  \frac{\phi ( xs^{-\alpha} )}{\phi(x)} \frac{ds}{s^{n+1}}
   \le \int_{2^{j+1}}^\fz  \frac{\phi ( xs^{-\alpha} )}{\phi(x)} \frac{ds}{s^{n+1}}
    , \end{align*}
  which implies that  $\phi(xs^{-\alpha})s^{-n}\to0$ as $s\to0$.
  Moreover, this also implies that
\begin{align*}\sup_{s\ge1} \frac{\phi ( xs^{-\alpha} )}{s^{ n}\phi(x)} &\le
 \sup_{j\ge0}\sup_{s\in[2^{j},2^{j+1})}  \frac{\phi ( xs^{-\alpha} )}{s^{ n}\phi(x)}  \le 2^{3n}\int_2^\fz  \frac{\phi ( xs^{-\alpha} )}{\phi(x)} \frac{ds}{s^{n+1}}
  \le 2^{3n}{\overline\Lambda_\phi(\alpha)} \quad\forall x>0, \end{align*}
 which gives \eqref{deltafzz}.
\end{proof}

\begin{lem}\label{l2.1}  Let $\alpha\in (-n,0)$  and $\phi$  be a Young  function.  For any domain $\Omega\subset\rn$,
we have ${\bf  B} ^{\alpha,\phi}(\Omega)\subset {\bf\dot B} ^{\alpha,\phi}(\Omega) \subset L^1_\loc(\Omega  )$ as sets.
\end{lem}
\begin{proof}
Let $u\in{\bf\dot B}^{\alpha,\phi} (\Omega)$ and
 $\lambda>\|u\|_{{\bf\dot B}  ^{\alpha,\phi}(\Omega)}$,
we have
$$\int_\Omega\int_\Omega \phi \left(\frac{|u(x)-u(y)|}{\lambda |x-y|^{\alpha}}\right)\,\frac{dydx}{|x-y|^{2n}}\le1.$$
By Fubini's theorem,  for almost all $x\in  \Omega$ we have
$$\int_\Omega \phi \left(\frac{|u(x)-u(y)|}{\lambda |x-y|^{\alpha}}\right)\,\frac{ dy}{|x-y|^{2n}}<\fz.$$
For any $B=B(z,r)$ with $z\in\Omega$ and $r<\frac13\diam\Omega$, choose an $x\in 3B\setminus 2B$ satisfying above inequality. Then $r\le |x-y|\le 4r$ for all $ y\in B$,  and hence
 $$\int_{B} \phi \left(\frac{|u(x)-u(y)|}{   \lambda r^{\alpha} }\right)\frac{dy}{(4r)^{2n}} \le \int_{B} \phi \left(\frac{|u(x)-u(y)|}{   \lambda |x-y|^{\alpha} }\right)\,\frac{dydx}{|x-y|^{2n}}  <\fz.$$
By Jessen's inequality, we have
 $$ \phi\left ( \fint_B \frac{|u(x)-u(y)|}{  \lambda r^{\alpha} }\, dy \right) <\fz $$
which implies that
    $ \fint_B  |u(x)-u(y)| \, dy<\fz$
  that is, $u\in L^1(B)$.  This completes the proof of Lemma \ref{l2.1}.
\end{proof}

\begin{lem}\label{l2.2}
Let  $ \Omega\subset\rn$ be a domain.
Assume that $\alpha\in (-n,0)$  and
 $\phi$  is a Young  function   satisfying \eqref{delta0}  and, when $\diam\Omega=\fz$, also satisfying \eqref{deltafz}.
 Then $C_c^1(\Omega)\subset {\bf  B} ^{\alpha,\phi}(\Omega)\subset {\bf\dot B} ^{\alpha,\phi}(\Omega) $ as sets.
\end{lem}
\begin{proof} Let $u\in C_c^1(\Omega)$. Obviously, $u\in L^\phi(\Omega)$.
  To see $u\in {\bf\dot B} ^{\alpha,\phi}(\Omega) $, let $V={\rm\,supp\,}u $, and   $W\subset \Omega$ be a bounded open set so that $V\subset W$. Then
\begin{align*}
H&=\int_\Omega\int_\Omega\phi\left(\frac{|u (z)-u (w)|}{\lambda|z-w|^{\alpha}}\right)\frac{dzdw}{|z-w|^{2n}}\\
&= \int_W\int_{W}\phi\left(\frac{|u (z)-u (w)|}{\lambda|z-w|^{\alpha}}\right)\frac{dzdw}{|z-w|^{2n}}+
2\int_{\Omega\setminus W}\int_{V}\phi\left(\frac{u(z)}{\lambda | z-w|^{ \alpha}} \right)\frac{dzdw}{|z-w|^{2n}}\\
&=:H_1+H_2.
\end{align*}
It then suffices to show that $H_1\le 1/2$ and $H_2\le 1/2$ when $\lambda$ is sufficiently large.
Write  $L=\|Du\|_{L^\fz(\Omega)}$. By \eqref{delta0} one has
\begin{align*}
 H_1 &\le \int_W\int_{B(w,  \diam  W) }\phi\left(\frac{| z-w|^{1-\alpha}}{\lambda/L} \right)\frac{dz}{|z-w|^{2n}}dw\\
&=\omega_n\int_W\int_0^{ \diam W  }\phi\left(\frac{t^{1-\alpha}}{\lambda/L} \right)\frac{ dt}{t^{n+1}}\,dw\\
&=\omega_n|W|( \diam  W )^{-n} \int_0^1\phi\left(\frac{ (\diam  W )^{1-\alpha}t^{1-\alpha}}{\lambda/L  } \right)\frac{ dt}{t^{n+1}}\\
&\le {\underline\Lambda_\phi(\alpha)} \omega_n|W|( \diam W )^{-n}\phi\left(\frac{( \diam W)^{1-\alpha}L  }{\lambda}\right),
\end{align*}
and hence  $ H_1\le 1/2 $ when $\lambda>0$ is large enough.    If $\Omega$ is bounded,   then
\begin{align*}
H_2\le 2 \dist(V,\partial W)^{-2n}|\Omega\setminus W||V|\phi\left(\frac{\|u\|_{L^\fz(\Omega)}( \diam\Omega)^{|\alpha|}}{\lambda  }\right),
\end{align*}
and hence  $ H_2\le 1/2 $ when $\lambda$ is large enough.
If $\Omega$ is unbounded, by \eqref{deltafz} we have
 \begin{align*}
H_2%\le 2  \int_{\Omega\setminus W}\int_{V}\phi\left(\frac{\|u\|_{L^\fz(\Omega)}}{\lambda | z-w|^{ \alpha}} \right)\frac{dzdw}{|z-w|^{2n}}
&\le 2\int_V \int_{\rn\setminus W}\phi\left(\frac{\|u\|_{L^\fz(\Omega)}}{\lambda | z-w|^{ \alpha}} \right)\frac{ dw}{|z-w|^{2n}}\,dz\\
&\le 2\omega_n\int_V \int_{\dist(V,\partial W)}^\fz\phi\left(\frac{\|u\|_{L^\fz(\Omega)}}{\lambda t^{ \alpha}} \right)\frac{ dt}{t^{ n+1}}\,dz\\
&\le 2\omega_n [\dist(V,\partial W)]^{-n}|V|
\int_{1}^\fz\phi\left(\frac{\|u\|_{L^\fz(\Omega)}}{\lambda [\dist(V,\partial W)]^\alpha t^{ \alpha}} \right)\frac{ dt}{t^{ n+1}} \\
&\le 2\omega_n [\dist(V,\partial W)]^{-n}|V| \overline\Lambda_\phi(\alpha) \phi\left(\frac{\|u\|_{L^\fz(\Omega)}}{\lambda [\dist(V,\partial W)]^\alpha  } \right),
\end{align*}
 hence  $ H_2\le 1/2 $ when  $\lambda$ is large enough.
 This completes the proof of Lemma \ref{l2.2}.
\end{proof}

\begin{rem}\label{r2.4}\rm
(i) The  assumption \eqref{delta0}   is optimal to guarantee $C^1_c(\Omega)\subset {\bf\dot B} ^{\alpha,\phi}(\Omega)$  and hence the non-triviality of ${\bf\dot B} ^{\alpha,\phi}(\Omega)$ in the following sense.
%
% The the  assumption \eqref{delta0} is optimal to guarantee the nontrivially of
%${\bf\dot B} ^{\alpha,\phi}(\Omega)$, ${\bf  B} ^{\alpha,\phi}(\Omega)$ and ${\bf\dot B}^{\alpha,\phi}_\ast(\Omega)$ in the following sense.
   For $\alpha\ge 1-n$ and $p\ge1$, by a direct calculation, %\le p\le n/(1-\alpha)$,
      the Young function  $ \phi(t)=t^p$   %and $t^{n/\alpha}\ln(1+t)$
 satisfies   \eqref{delta0} if and only if $p>n/(1-\alpha)$.
Note that  ${\bf\dot B}^{\alpha,\phi}(\Omega)=
{\bf\dot B}^{n/p+\alpha}_{pp}(\Omega)$. By \cite[Theorems 4.1\&4.2]{gkz13},  ${\bf\dot B}^{n/p+\alpha}_{pp}(\Omega)$,
and hence ${\bf\dot B}^{\alpha,\phi}(\Omega)$, is  nontrivial (or contains  $C^1_c(\Omega)$) if and only if $s=n/p+\alpha<1$, that is,  $p>n/(1-\alpha)$.

(ii) In the case $\diam\Omega=\fz$, \eqref{deltafz} is optimal to guarantee
$C^1_c(\Omega)\subset {\bf\dot B} ^{\alpha,\phi}(\Omega)$ or the nontriviality of ${\bf\dot B} ^{\alpha,\phi}(\Omega)$ in the following sense.
 Indeed,  the Young function $ t^p$ with $p\ge1$ satisfies \eqref{deltafz} if and only if $p<n/|\alpha|$.
 Let $\phi(t)=t^{n/|\alpha|}$ and $\Omega$ be any unbounded globally $n$-regular domain.  We see that
$\|u\|_{{\bf\dot B} ^{\alpha,\phi}(\Omega)}=\fz$ for any $u\in C_c^1(\Omega)$ and $u\not\equiv0$.
Indeed, for any $\lambda>0$, let $H$ and $H_2$ be as in the proof of Lemma \ref{l2.2}.
Moreover, let   $V_0=\{x\in V: |u(x)|> \|u\|_{L^\fz(\Omega)}/2\} $
 and  $W=B(z_0,4\diam V)$ with $z_0\in V$. For all $w\in\Omega\setminus W$ and $z\in V$, $|z-z_0|\le \diam V, |z_0-w|\ge 4\diam V$, it deduces that $|z-z_0|\le\frac14|z_0-w|$. Then $|z-w|\le |z_0-w|+|z-z_0|\le  \frac54 |z_0-w|.$
 Thus
\begin{align*}
H_2%&\ge    \int_{\Omega\setminus W}\int_{V_0} \frac{\|u\|^{n/\alpha}_{L^\fz(\Omega)}}{(2\lambda)^{n/|\alpha|}  } \frac{dzdw}{|z-w|^{ n}}\\
&\ge  \frac{\|u\|^{n/|\alpha|}_{L^\fz(\Omega)}}{(2\lambda)^{n/|\alpha|}  }  \int_{V_0} \int_{\Omega\setminus W} \frac{dzdw}{|z-w|^{ n}}\ge\frac{\|u\|^{n/|\alpha|}_{L^\fz(\Omega)}}{(2\lambda)^{n/|\alpha|} (\frac54)^{n} }|V_0|  \int_{\Omega\setminus W} \frac{ dw}{|z_0-w|^{ n}}.
\end{align*}
But
\begin{align*}\int_{\Omega\setminus W} \frac{ dw}{|z_0-w|^{ n}} &  \ge  \sum_{j\ge 2} 2^{-jk_0n}(\diam V)^{-n}|\Omega\cap [B(z_0, 2^{(j+1)k_0}\diam V)\setminus B(z_0, 2^{jk_0}\diam V)]|\\
&\ge  \sum_{j\ge 2} 2^{-jk_0n} 2^{jk_0n } =\fz,\end{align*}
where    $k_0$ is a postiche integer satisfying $\theta 2^{ k_0n}\ge 2\omega_n$ so that
\begin{align*} |\Omega\cap [B(z_0,2^{i+k_0}\diam V)\setminus B(z_0,2^{i }\diam V)]|
&= |\Omega\cap  B(z_0,2^{i+k_0}\diam V)|- |\Omega\cap B(z_0,2^{i }\diam V) |\\
&\ge  \theta 2^{ k_0n} 2^{in }(\diam V)^n- \omega_n 2^{in }(\diam V)^n\ge 2^{in }(\diam V)^n.\end{align*}
Therefore, we always $H\ge H_2=\fz$, which means that $u\notin {\bf\dot B} ^{\alpha,\phi}(\Omega)$.
\end{rem}

To end this section, we calculate  $ {\bf B} ^{\alpha,\phi}(\Omega)$ and
${\bf\dot B} ^{\alpha,\phi}(\Omega)$-norms of some special functions, which will be used in Sections 3 and 4.

For $x\in \Omega$ and $0<r<t< \frac12\diam\Omega$, let $B_\Omega(x,t):=\Omega\cap B(x,t)$ and $B_\Omega(x,r):=\Omega\cap B(x,r)$, set  the function
\begin{eqnarray}\label{defu}
u_{x,r,t}(z)=\left\{\begin{array}{ll}1& z\in B_\Omega(x,r);\\
\frac{t-|x-z|}{t-r}\quad & z\in B_\Omega(x,t)\setminus B_\Omega(x,r);\\
0& z\in\Omega\setminus B_\Omega(x,t).
\end{array}\r.
\end{eqnarray}

\begin{lem}\label{l2.3}
Let $\alpha\in (0,n)$  and $\phi  $  be a Young function satisfying \eqref{delta0} and \eqref{deltafz}. Then there exists a constant $C>0$ depending on $n,\alpha$ and $\phi$ such that  for any domain $\Omega\subset\rn$,
  $x\in \Omega$ and $0<r<t<\frac12\diam\Omega$,  we have $u_{x,r,t}\in {\bf\dot B} ^{\alpha,\phi}(\Omega)$  and
$$\|u_{x,r,t}\|_{{\bf\dot B} ^{\alpha,\phi}(\Omega)}\le C (t-r)^{-\alpha}\left [\phi^{-1}\left(\frac{(t-r)^n}{|B_\Omega(x,t)| }\right)\right]^{-1}.$$
\end{lem}

\begin{proof}
Write $u= u_{x,r,t}$ for simple.
Note that
\begin{align*}
H&:=\int_\Omega\int_\Omega\phi\left(\frac{|u (z)-u (w)|}
{\lambda|z-w|^{\alpha}}\right)\frac{dzdw}{|z-w|^{2n}}\\
&=
\int_{B_\Omega(x,t)}\int_{B_\Omega(x,t)}
\phi\left(\frac{|u (z)-u (w)|}{\lambda|z-w|^{\alpha}}\right)
\frac{dzdw}{|z-w|^{2n}} +2\int_{\Omega\setminus B_\Omega(x,t)}\int_{B_\Omega(x,t)}\phi\left(\frac{|  u (z)|}{\lambda|z-w|^{\alpha}}\right)\frac{dzdw}{|z-w|^{2n}} \\
&\le \int_{B_\Omega(x,t)}\int_{B_\Omega(w,t-r)}
\phi\left(\frac{|z-w|^{1-\alpha}}{\lambda(t-r)}\right)\frac{dz}{|z-w|^{2n}}dw\\
& \quad +
\int_{B_\Omega(x,t)}\int_{B_\Omega(w,2t)\setminus B_\Omega(w,t-r)}\phi\left(\frac{ |z-w|^{-\alpha}}{\lambda}\right)\frac{dz}{|z-w|^{2n}}dw\\
&\quad+ 2\int_{B_\Omega(x,t) \setminus B_\Omega(x,r)}\int_{\Omega\setminus B_\Omega(x,t)}\phi\left(\frac{(t-|z-x|)|z-w|^{-\alpha}}{\lambda(t-r)}\right)
\frac{dw}{|z-w|^{2n}}dz \\
& \quad + 2\int_{B_\Omega(x,r) }
\int_{\Omega\setminus B_\Omega(x,t)}\phi\left(\frac{|z-w|^{-\alpha}}{\lambda }\right)\frac{dw}{|z-w|^{2n}}dz\\
&=:H_1+H_2+2H_3+2H_4.
\end{align*}

By \eqref{delta0}, we have
\begin{align*} %\int_{B_\Omega(w,t-r)}\phi\left(\frac{|z-w|^{\alpha+1}}{\lambda(t-r)}\right)\frac{dz}{|z-w|^{2n}}
H_1&\le  \omega_n|B_\Omega(x,t )| \int_0^{t-r}  \phi\left(\frac{s^{1-\alpha}}{\lambda(t-r)}\right)\frac{ds}{s^{n+1}} \\
&
 = \omega_n \frac{|B_\Omega(x,t )|}{(t-r)^{n}} \int_0^{1}  \phi\left(\frac{s^{1-\alpha}(t-r)^{-\alpha}}{\lambda}\right)\frac{ds}{s^{n+1}}\le \omega_n{\underline\Lambda_\phi(\alpha)} \frac{|B_\Omega(x,t )|}{(t-r)^{n}} \phi\left(\frac{(t-r)^{-\alpha} }{\lambda }\right) .
\end{align*}
By \eqref{deltafz} we have
\begin{align*}
% \int_{B_\Omega(w,2t)\setminus B_\Omega(w,t-r)}\phi\left(\frac{ |z-w|^{\alpha}}{\lambda}\right)  \frac{dz}{|z-w|^{2n}}
H_2&\le \omega_n  |B_\Omega(x,t )| \int_{t-r}^{\infty} \phi\left(\frac{s^{-\alpha}}{\lambda}\right) \frac{ds}{s^{n+1}}\\
&
 =\omega_n  \frac{|B_\Omega(x,t )|}{(t-r)^{n}} \int_{1}^{\infty} \phi\left(\frac{(t-r)^{-\alpha}s^{-\alpha}}{\lambda}\right) \frac{ds}{s^{n+1}}
\le   \omega_n  {\overline\Lambda_\phi(\alpha)} \frac{|B_\Omega(x,t )|}{(t-r)^{n}}\phi\left(\frac{ (t-r)^{-\alpha}}{\lambda}\right).
\end{align*}

 For any $z\in B_\Omega(x,t)\setminus B_\Omega(x,r)$, note that
 $\mbox{$ \Omega\setminus B_\Omega(x,t)\subset \Omega\setminus B_\Omega(z, t-|z-x|)$
 and $0<t-|z-x|<t-r$
 }.$
  By \eqref{deltafz} and \eqref{delta00} we have
\begin{eqnarray*}
\int_{\Omega\setminus B_\Omega(x,t)}\phi\left(\frac{(t-|z-x|)|z-w|^{-\alpha}}{\lambda(t-r)}\right)\frac{dw}{|z-w|^{2n}}&&\le
  \omega_n \int_{  t-|z-x|}^{\infty}  \phi\left(\frac{(t-|z-x|)s^{-\alpha}}{\lambda(t-r)}\right)\frac{ds}{s^{n+1}}\\
 &&\le   {\overline\Lambda_\phi(\alpha)}  \omega_n\frac{ 1 }{(t-|z-x|)^{n}}\phi\left(\frac{(t-|z-x|)^{1-\alpha}}{\lambda(t-r)}\right)\\
 &&
   \le 2^{2n} {\overline\Lambda_\phi(\alpha)}  {\underline\Lambda_\phi(\alpha)}\omega_n  \frac{1}{(t-r)^n}\phi\left(\frac{(t-r)^{-\alpha} }{\lambda/2^{1-\alpha}}\right)
 \end{eqnarray*}
 and hence,
\begin{align*}H_3
&\le 2^{2n} {\overline\Lambda_\phi(\alpha)}    {\underline\Lambda_\phi(\alpha)}\omega_n \frac{|B_\Omega(x,t)|}{(t-r)^n}\phi\left(\frac{(t-r)^{-\alpha} }{\lambda/2^{1-\alpha}}\right).
\end{align*}

Moreover, for any $z\in  B_\Omega(x,r)$, note that
  $\mbox{$ \Omega\setminus B_\Omega(x,t)\subset \Omega\setminus B_\Omega(z, t-|z-x|)$
 and $ t-|z-x|>t-r$
 }.$
 By \eqref{deltafz} and \eqref{deltafzz}
we obtain
 \begin{align*}\int_{\Omega\setminus B_\Omega(x,t)}\phi\left(\frac{|z-w|^{-\alpha}}{\lambda }\right)\frac{dw}{|z-w|^{2n}}&\le  {\overline\Lambda_\phi(\alpha)}  \omega_n \int_{  t-|z-x|}^{\infty}\phi\left(\frac{s^{-\alpha}}{\lambda}\right)\frac{dw}{s^{n+1}}\\
 &\le  {\overline\Lambda_\phi(\alpha)}  \omega_n\frac{  1 }{(t-|z-x|)^{n}}\phi\left(\frac{(t-|z-x|)^{-\alpha}}{\lambda }\right)
 \\&\le  2^{3n} {\overline\Lambda_\phi(\alpha)} ^{2 } \omega_n\frac{ 1 }{(t-r)^{n}}\phi\left(\frac{(t-r)^{-\alpha} }{\lambda}\right) .
 \end{align*}
and hence
\begin{align*}H_4%
&\le 2^{3n} {\overline\Lambda_\phi(\alpha)} ^{2 } \omega_n  \frac{|B_\Omega(x,t)|}{(t-r)^n}\phi\left(\frac{(t-r)^{-\alpha} }{\lambda }\right).
\end{align*}

Combining all above estimates together we conclude
 \begin{align*}
H\le M \frac{|B_\Omega(x,t)|}{(t-r)^n} \phi\left( \frac{(t-r)^{-\alpha} } {\lambda/2^{1-\alpha}}  \right),
\end{align*}
where  $M  =
 {\underline\Lambda_\phi(\alpha)}\omega_n + {\overline\Lambda_\phi(\alpha)}   \omega_n+ 2^{2n+1} {\overline\Lambda_\phi(\alpha)}  {\underline\Lambda_\phi(\alpha)}\omega_n+ 2^{3n+1} {\overline\Lambda_\phi(\alpha)} ^{2 }       \omega_n + 1 $.
Letting $$\lambda_0=
2^{1-\alpha}(t-r)^{-\alpha}M \left [\phi^{-1}\left(\frac{(t-r)^n}{|B_\Omega(x,t)| }\right)\right]^{-1}>0,$$  by $M \ge1$ and the convexity of $\phi$,
we have $\forall \lambda\ge \lambda_0$,
$$H \le
M \frac{|B_\Omega(x,t)|}{(t-r)^n} \phi\left(\frac1{M } \phi^{-1}\left(\frac{(t-r)^n}{|B_\Omega(x,t)| }\right)  \right)
\le 1.$$
Thus $\|u_{x,r,t}\|_{{\bf\dot B} ^{\alpha,\phi}(\Omega)}\le \lambda_0$ as desired.
\end{proof}

\begin{lem}\label{l2.4}
Let $\alpha\in (-n,0)$  and $\phi  $  be a Young function satisfying \eqref{delta0} and \eqref{deltafz}.
Assume  $\Omega\subset\rn$ is a bounded domain.
There exists a constant $C>0$ depending on $n,\alpha$, $\Omega$ and $\phi$ such that
  $x\in \Omega$ and $0<r<t<\frac12 \diam\Omega$,    $u_{x,r,t}\in {\bf B} ^{\alpha,\phi}(\Omega)$  and
$$\|u_{x,r,t}\|_{ {\bf B} ^{\alpha,\phi}(\Omega)}\le C (t-r)^{-\alpha}\left [\phi^{-1}\left(\frac{(t-r)^n}{|B_\Omega(x,t)| }\right)\right]^{-1}.$$
\end{lem}

\begin{proof}
Note that
  $$\|u_{x,r,t}\|_{L ^{ \phi}(\Omega)}\le \left[\phi^{-1}\left(\frac1{|B_\Omega(x,t)| }\right)\right]^{-1}.$$
  Indeed, for $\lambda> [\phi^{-1}(\frac1{|B_\Omega(x,t)| })]^{-1}$,
 since $u_{x,r,t}$ is supported in $B_\Omega(x,t)$ and $0\le   u_{x,r,t}\le1$,
 we have $$\int_\Omega\phi\left( \frac{u_{x,r,t}}\lambda\right)\le \phi(1/\lambda)
 |B_\Omega(x,t)|<1 $$ as desired.

  It then suffices to prove that   there exists a constant $C>0$ such that
  $$  \left[\phi^{-1}\left(\frac1{|B_\Omega(x,t)| }\right)\right]^{-1}\le C (t-r)^{-\alpha}\left [\phi^{-1}\left(\frac{(t-r)^n}{|B_\Omega(x,t)| }\right)\right]^{-1}.$$
Note that this is equivalent to
   % $$ \phi^{-1}\left(\frac{(t-r)^n}{|B_\Omega(x,t)| }\right) \le \phi^{-1}(\frac1{|B_\Omega(x,t)| })  C{(t-r)^{\alpha}}  ,$$
           $$  \frac{(t-r)^n}{|B_\Omega(x,t)| }  \le \phi\left(\phi^{-1}\left(\frac1{|B_\Omega(x,t)| }\right)  C{(t-r)^{-\alpha}}\right)  ,$$
   Since $\phi(x { (\diam\Omega) ^{-\alpha}})\le 2^{3n} {\overline\Lambda_\phi(\alpha)}(\frac{\diam\Omega}{t-r})^n\phi(x {(t-r)^{-\alpha}})   $ as in \eqref{deltafzz},
  it suffices to prove
  $$  \frac{ {2^{3n}\overline\Lambda_\phi(\alpha)} (\diam\Omega) ^n}{|B_\Omega(x,t)| }  \le  \phi\left(\phi^{-1}\left(\frac1{|B_\Omega(x,t)| }\right)  C{( \diam\Omega)^{-\alpha}}\right)  ,$$
  Choosing $C>( \diam\Omega)^{\alpha} ( {2^{3n}\overline\Lambda_\phi(\alpha)}( \diam\Omega)^n+1)$,
   by the convexity of $\phi$ we have
  \begin{align*}\phi\left(\phi^{-1}\left(\frac1{|B_\Omega(x,t)| }\right)  C{( \diam\Omega)^{-\alpha}}\right)
  &\ge \phi\left(\phi^{-1}\left(\frac1{|B_\Omega(x,t)| }\right) \right)  [C ( \diam\Omega)^{-\alpha}]\\
  &
  =\frac{C (\diam\Omega)^{-\alpha}}{|B_\Omega(x,t)|}\ge  \frac{ {2^{3n}\overline\Lambda_\phi(\alpha)}(\diam\Omega)^n}{|B_\Omega(x,t)| }
  \end{align*}
  as desired.
\end{proof}

\section{Proofs of Theorems \ref{mainthm0}\&\ref{mainthm} and Corollary \ref{c1.2}}

%The proofs are partially motivated by \cite{xx,sv,cv} and also \cite{hkt08,z}, for details and more remarks see Remark 3.1 at the end of this section.

%To show that Ahlfors $n$-regular domains support  ${\bf\dot B}^{\alpha,\phi}$-imbedding, we need the following lemmas.

We begin with the proof of Lemma \ref{l2.7x}, which is motivated by \cite{cv,sv13} and also \cite{dpv}.

\begin{proof}[Proof of Lemma \ref{l2.7x}.]
  Let $\kz = [2 \omega_n/\theta ]^{1/n}+2$.
  Then
 \begin{eqnarray}\label{e3.2}
 \Omega\cap( B(z, \kz  s)\setminus  B(z,  s)) \ne\emptyset\  \forall z\in\Omega\ \mbox{and}\
  0<s<\frac2\kz\diam\Omega.
  \end{eqnarray}
Indeed, we have
$|\Omega\cap B(z, \kz  s)|\ge \theta  \kz ^n s^n$ and
and $|\Omega\cap B(z,  s)|\le \omega_n   s^n$    for all $z\in\Omega$ and
  $0<s<\frac2\kz\diam\Omega .$
Since
  $\theta  \kz ^n>2 \omega_n  $, we know that  $\Omega\cap( B(z, \kz  s)\setminus  B(z,    s))$ has positive measure.

Let $r\in(0,2\diam\Omega)$ such that
  $|E|=  | \Omega\cap B(x,r)|$,
  and moreover, $\theta r^n\le |E|\le \omega_n r^n.$
If $r\ge \frac1{8\kz}\diam\Omega$, then $|E|\ge C|\Omega|$ and $\diam\Omega<\infty$.
By \eqref{deltafzz},    for all $y\in \Omega,$  we have
$$\phi(t|x-y|^{-\alpha})%\ge \phi(t(2r)^\alpha (|x-y|/2r)^{\alpha})
\ge \frac1{2^{3n}\overline \Lambda_\phi(\alpha)}\phi(t(  \diam\Omega)^{-\alpha}) \left(\frac{|x-y|}{  \diam\Omega}\right)^{n}\ge  \frac1{2^{3n}\overline \Lambda_\phi(\alpha)}\phi(t(  \diam\Omega)^{-\alpha}) \left(\frac{|x-y|}{  \diam\Omega}\right)^{2n},$$
and hence there exist  positive constants $C_1$  and $C_2$  such that
$$
\int_{\Omega\setminus E  }\frac{\phi(t|x-y|^{-\alpha})}{|x-y|^{2n}}dy \ge C_1
\frac{|\Omega\setminus E|}{|\Omega|^2} \phi( t|\Omega|^{-\alpha/n}) \ge  C_1C
\frac1{|E|}\frac{|\Omega\setminus E|}{|\Omega| } \phi( C_2t|E|^{-\alpha/n})
$$
as desired.

If $r<\frac1{8\kz}\diam\Omega$, write
\begin{eqnarray*}
\int_{ \Omega\setminus E }\frac{\phi(t|x-y|^{-\alpha})}{|x-y|^{2n}}dy&&
=\int_{( \Omega\setminus E)   \cap B(x,r) } \frac{\phi(t|x-y|^{-\alpha})}{|x-y|^{2n}}dy+
\int_{( \Omega\setminus E) \setminus  B(x,r)  } \frac{\phi(t|x-y|^{-\alpha})}{|x-y|^{2n}}dy.
\end{eqnarray*}
By \eqref{deltafzz}, for $y\in B(x,r)$  we have
$$\phi(t|x-y|^{-\alpha})\ge \frac{1}{2^{3n}{\overline\Lambda_\phi(\alpha)}}\phi(tr^{-\alpha})
\left(\frac{|x-y|}{r}\right)^n\ge \frac{1}{2^{3n}{\overline\Lambda_\phi(\alpha)}}\phi(tr^{-\alpha})
\left(\frac{|x-y|}{r}\right)^{2n}.$$
Thus
\begin{align*}\int_{( \Omega\setminus E) \cap B(x,r) } \frac{\phi(t|x-y|^{-\alpha})}{|x-y|^{2n}}dy
 &\ge \frac{1}{2^{3n}{\overline\Lambda_\phi(\alpha)}}\frac{ \phi(t r^{-\alpha})}{r^{2n}}|( \Omega\setminus E)\cap B(x,r)|.
 \end{align*}
Note that
$$| ( \Omega\setminus E) \cap B(x,r)|=| \Omega\cap B(x,r) |-|E\cap B(x,r)|=|E|-|E\cap B(x,r)|=|E\setminus  B(x,r) |.$$
By  \eqref{deltafzz},  for $y \in E\setminus B(x,r)$ we have
$$\phi(tr^{-\alpha})\ge
\frac{1}{2^{3n}{\overline\Lambda_\phi(\alpha)}}
\phi(t|x-y|^{-\alpha})\left(\frac{r}{|x-y|}\right)^n\ge \frac{1}{2^{3n}{\overline\Lambda_\phi(\alpha)}}
\phi(t|x-y|^{-\alpha})\left(\frac{r}{|x-y|}\right)^{2n}.$$
Therefore,
\begin{align*}\int_{( \Omega\setminus E) \cap B(x,r) } \frac{\phi(t|x-y|^{-\alpha})}{|x-y|^{2n}}dy&
 \ge\left(\frac{1}{2^{3n}{\overline\Lambda_\phi(\alpha)}}\right)^2 \int_{E \setminus B(x,r) } \phi(t|x-y|^{-\alpha})\frac{1}{|x-y|^{2n}} dy.
   \end{align*}
 Since $[(  \Omega\setminus E)\setminus B(x,r)]\cup [E \setminus B(x,r)]= \Omega\setminus B(x,r)$, we obtain

\begin{eqnarray*}
\int_{\Omega\setminus E  }\phi(t|x-y|^{-\alpha})\frac{1}{|x-y|^{2n}}dy&&\ge \min\left\{\left(\frac{1}{2^{3n}{\overline\Lambda_\phi(\alpha)}}\right)^2,1\right\} \int_{\Omega\setminus B(x,r)  } \phi(t|x-y|^{-\alpha})\frac{1}{|x-y|^{2n}}dy.
\end{eqnarray*}
By \eqref{e3.2},
$\Omega\cap (B(x,2\kz r )\setminus B(x,2r))$ is not empty set, and hence containing some point, say $z$.
Then $$\Omega\cap B(z,r)\subset\Omega\cap [
(B(x,3\kz r )\setminus B(x,r)]\subset \Omega\setminus B(x,r)$$
and $$|\Omega\cap
[(B(x,3\kz r )\setminus B(x,r)]| \ge |B(z,r)\cap \Omega|\ge \theta  r^n.$$
Thus, we have
\begin{align*}
\int_{\Omega\setminus B(x,r)  } \phi(t|x-y|^{-\alpha})\frac{1}{|x-y|^{2n}}dy
&\ge   \int_{\Omega\cap
[B(x,3\kz r )\setminus B(x,r)] } \phi(t|x-y|^{-\alpha})\frac{1}{|x-y|^{2n}}dy \ge  C_3r^{-n} \phi( tr^{-\alpha}),
\end{align*}
where $C_3$ is positve constant.
Note that  $\theta r^n\le |E|\le \omega_n r^n$  and $\frac{|\Omega \setminus E|}{|\Omega|}\le 1$. The proof of Lemma \ref{l2.7x} is completed.
\end{proof}

 \begin{lem}\label{l3.2}
Let $\alpha\in (-n,0)$  and $\phi $  be a Young function satisfying \eqref{delta0} and \eqref{deltafz}. Suppose that $\Omega$ is a   globally $n$-regular domain.
Then there exists a constant $C>0$ depending only on $n, \alpha,\phi$ and $\theta$ such that
$$\|u\|_{L^{n/|\alpha|}(\Omega)}\le C\|u\|_{{ {\bf\dot B}}^{\alpha,\phi}(\Omega)}\quad
%\mbox{whenever $u \in {\bf B}^{\alpha,\phi}(\Omega)$    satisfies \eqref{condu}and \eqref{condufz}.}
$$
    whenever
    $u \in {\bf{ \dot B}}^{\alpha,\phi}(\Omega)$ satisfies
\begin{equation}\label{condu}
\mbox{$u\ge0 $ in $\Omega$ and $|\{x\in \Omega, u=0\}|\ge \frac12|\Omega|$\ {\rm if} $|\Omega|<\fz$}
\end{equation}
or
\begin{equation}\label{condufz}
\mbox{$u\ge0 $ in $\Omega$ and $|\{x\in \Omega, u>a\}|<\fz  \ \forall a>0$\ {\rm  if} $|\Omega|=\fz$}.
\end{equation}
  \end{lem}

  \begin{proof}

  Let $u \in {\bf \dot B}^{\alpha,\phi}(\Omega)$
    satisfy \eqref{condu} or \eqref{condufz}.
    Obviously, we may assume that $u\not \equiv0$.
Without loss of generality, we may also assume that $u$ is bounded.
Indeed, for $N\ge 1$ let $$u^N=u\chi_{\{x\in \Omega: u<2^N\}}+ 2^N \chi_{\{x\in \Omega: u\ge 2^N\}}.$$
Note that $\|u \|_{L^{n/|\alpha|}(\Omega)}=\lim_{N\to\fz}\|u^N\|_{L^{n/|\alpha|}(\Omega)}$
and $\sup_{N\in\nn}\|u^N\|_{{\bf\dot B}^{\alpha,\phi}(\Omega)}\le \|u \|_{{\bf \dot B}^{\alpha,\phi}(\Omega)}$.
If  $\|u^N\|_{L^{n/|\alpha|}(\Omega)}\le C\|u^N\|_{{\bf \dot B}^{\alpha,\phi}(\Omega)} $ hold
for all $N\ge1$, by sending $N\to\infty$, we have
 $\|u \|_{L^{n/|\alpha|}(\Omega)}\le C\|u \|_{{\bf \dot B}^{\alpha,\phi}(\Omega)} $ as desired.
Moreover, when $|\Omega|=\fz$, we may further assume $|\{x\in \Omega, u>0\}|<\fz  $. Indeed,
for $N\le 0$ let $$u_N=(u-2^N)\chi_{\{x\in \Omega: u(x)>2^N\}}.$$
By \eqref{condufz}, we have
$$|\{x\in \Omega, u_N(x)>0\}|= |\{x\in \Omega, u(x)>2^N\}|<\fz  .$$
Note that
 $\|u \|_{L^{n/|\alpha|}(\Omega)}=\lim_{N\to\fz}\|u_N\|_{L^{n/|\alpha|}(\Omega)}$
and $\sup_{N\in\nn}\|u_N\|_{{\bf\dot B}^{\alpha,\phi}(\Omega)}\le \|u \|_{{\bf \dot B}^{\alpha,\phi}(\Omega)}$.
 If  $\|u_N\|_{L^{n/|\alpha|}(\Omega)}\le C\|u_N\|_{{\bf\dot B}^{\alpha,\phi}(\Omega)} $ hold
for all $N\le0$, by sending $N\to-\infty$, we have $u\in L^{n/|\alpha|}(\Omega)$
and  $\|u \|_{L^{n/|\alpha|}(\Omega)}\le C\|u \|_{{\bf B}^{\alpha,\phi}(\Omega)} $ as desired.

Under above assumptions on $u$, we have $ u\in L^{n/|\alpha|}(\Omega)$.
Indeed, in the case $|\Omega|<\fz$, the   boundedness    of $u $ implies that
 $ u\in L^{n/|\alpha|}(\Omega)$.
In the case $|\Omega|=\fz$,
 the assumption
$|\{x\in\Omega: u(x)> 0\}|<\fz$ and
  the  boundedness of $u$ also gives
 $ u\in L^{n/|\alpha|}(\Omega)$.
  Write $$A_k:=\{z\in B: u(z)>2^k\}\quad{\rm and }\quad D_k:=A_k\setminus A_{k+1}=\{z\in B: 2^k<u(z) \le 2^{k+1}\},$$
  and $a_k :=|A_k| $ and    $ d_k :=|D_k|$ for $k\in\zz$.
%Also set $D_{-\fz}:=A_{-\fz}:=\{x\in\Omega: u(z)=0\}$.
%%$$  \|u\|^{n/|\alpha|}_{L^{n/|\alpha|}(\Omega)}\le
%%\sum_{i\in\zz} d_i 2^{(i+1)n/|\alpha|}\le 2^{n/|\alpha|}\sum_{i\in\zz} a_{i } 2^{in/|\alpha|} $$
%%and
%\begin{eqnarray}\label{ev2.x3} \sum_{i\in\zz} a_{i } 2^{in/|\alpha|}&&\le
%\sum_{i\in\zz} \sum_{j\ge i}d_{j } 2^{in/|\alpha|} \\
%&&=\sum_{j}d_j\sum_{i\le j}2^{in/|\alpha|} \le 2^{ n/|\alpha|}\sum_{j}d_j2^{jn/|\alpha|}\nonumber\\
%&&\le  2^{ n/|\alpha|}\|u\|^{n/|\alpha|}_{L^{n/|\alpha|}(\Omega)}\nonumber \le
%2^{ n/|\alpha|}\sum_{i\in\zz} d_i 2^{(i+1)n/|\alpha|}\nonumber \le 2^{2n/|\alpha|}\sum_{i\in\zz} a_{i } 2^{in/|\alpha|}\nonumber
%\end{eqnarray}
Then
$$  \|u\|^{n/|\alpha|}_{L^{n/|\alpha|}(\Omega)}\le
\sum_{i\in\zz} d_i 2^{(i+1)n/|\alpha|}\le 2^{n/|\alpha|}\sum_{i\in\zz} a_{i } 2^{in/|\alpha|} $$
and
\begin{equation}\label{ev2.x3} \sum_{i\in\zz} a_{i } 2^{in/|\alpha|}\le
\sum_{i\in\zz} \sum_{j\ge i}d_{j } 2^{in/|\alpha|}=\sum_{j}d_j\sum_{i\le j}2^{in/|\alpha|}\le \frac{1}{1-2^{n/\alpha}}\sum_{j}d_j2^{jn/|\alpha|}\le \frac{1}{1-2^{n/\alpha}}\|u\|^{n/|\alpha|}_{L^{n/|\alpha|}(\Omega)}.
\end{equation}

On the other hand,
observe that $\{D_l\}_{l\in\zz}$, and hence
$\{D_l\times (\Omega\setminus A_{l-1})\}_{l\in\zz}$,  are disjoint for each other,
and that for any $(x,y)\in D_l\times (\Omega\setminus A_{l-1})$,
we have
$u(x)\ge 2^l$ and $ u(y)\le 2^{l-1}$,  and hence  $|u(x)-u(y)|\ge 2^{l-1}$.
Therefore,
\begin{eqnarray}\label{e2.22}
 H: =\iint_{\Omega\times \Omega}\phi\left(\frac{|u(x)-u(y)|}{\lambda|x-y|^{\alpha}}\right)
 \frac{dxdy}{|x-y|^{2n}}
 \ge  \sum_{l\in\zz}\int_{D_l}
 \int_{ \Omega\setminus A_{l-1} } \phi\left(\frac{2^{l-1}}{\lambda|x-y|^{\alpha}}\right)\frac{dy}{|x-y|^{2n}}dx
\nonumber.
\end{eqnarray}
 If $|\Omega|<\fz$, by \eqref{condu}, we know that $a_k \le \frac12|\Omega|$ for all $k\in\zz$.
If $|\Omega|=\fz$, then, by \eqref{condufz} we have $a_l<\fz$ for all $l\in\zz$.
Thus, applying Lemma \ref{l2.7x}, we obtain
\begin{equation}\label{ev2.1}
H  \ge  C_1\sum_{l\in\zz,a_{l-1}\ne0}\frac{d_l}{a_{l-1}}  \frac{|\Omega\setminus A_{l-1}|}{|\Omega|} \phi\left(\frac{C_22^{l-1}a_{l-1}^{|\alpha|/n} }{\lambda }\right)  \ge \frac12 C_1\sum_{l\in\zz,a_{l-1}\ne0}\frac{d_l}{a_{l-1}}  \phi\left(\frac{C_22^{l-1}a_{l-1}^{|\alpha|/n} }{\lambda }\right).\end{equation}

Let $\lambda =    M(\sum_{i\in\zz} a_{i } 2^{in/|\alpha|})^{|\alpha|/n}$ with
$$M =  C_2\left[\phi^{-1}\left(\frac{2^{3n+2}2^{ n/|\alpha|}\overline\Lambda_\phi(\az) }{C_1(1-2^{n/\alpha})}\right) \right]^{-1} .$$
Noting $\phi(sC_2/M)\ge 2^{-3n }[\overline\Lambda_\phi(\alpha)]^{-1}\phi(C_2/M)  s^{n/|\alpha|}$ for any $s\in(0,1)$ as given in \eqref{deltafzz}, by  \eqref{ev2.1}  we obtain
\begin{eqnarray*}
H
&&\ge\frac{C_1}{2} \sum_{\substack{l\in \zz , \, a_{l-1}\neq0}}\frac{d_l}{a_{l-1}}
\phi\left(\frac{C_2}{M } \frac{ 2^{l-1}a_{l-1}^{|\alpha|/n} }{  (  \sum_{i\in \zz }  2^{i n/|\alpha|} a_i)^{|\alpha|/n}  }\right) \\
&&\ge\frac{C_1}{2^{3n+1 }\overline\Lambda_\phi(\alpha) }\sum_{\substack{l\in \zz , \, a_{l-1}\neq0}}\frac{d_l}{a_{l-1}} \frac{2^{(l-1)n/|\alpha|}a_{l-1}}
{  \sum_{i\in \zz }  2^{i n/|\alpha|} a_{i }}\phi\Big(\frac{C_2}{M }\Big)\\
&&\ge\frac{C }{2^{3n +1}\overline\Lambda_\phi(\alpha) } \frac{\sum_{\substack{l\in \zz , \, a_{l-1}\neq0}} d_l  2^{(l-1)n/|\alpha|} }
{  \sum_{i\in \zz }  2^{i n/|\alpha|} a_{i }}\phi\Big(\frac{C_2}{M }\Big ). \\
\end{eqnarray*}
Since $a_{l-1}=0$ implies that $d_l=0$, we have
$\sum_{\substack{l\in \zz,a_{l-1}\ne0 }} d_l  2^{(l-1)n/|\alpha|}= \sum_{\substack{l\in \zz }} d_l  2^{(l-1)n/|\alpha|}$.
Thus by \eqref{ev2.x3},
\begin{eqnarray*}
H
&&\ge\frac{C_1(1-2^{n/\alpha}) }{2^{3n+1 }2^{ n/|\alpha|}\overline\Lambda_\phi(\alpha) }\phi\Big(\frac{C_2}{M } \Big)=2,
\end{eqnarray*}
 which  implies that
  $M\left ( \sum_{i\in\zz} a_{i } 2^{in/|\alpha|}\right)^{|\alpha|/n}\le \|u\|_{{\bf\dot B}^{\alpha,\phi}(\Omega)},$
  that is, $\|u\| _{L^{n/|\alpha|}(\Omega)} \le 2 M ^{-1} \|u\|_{{\bf \dot B}^{\alpha,\phi}(\Omega)}$
as desired. The proof of Lemma \ref{l3.2} is completed.
  \end{proof}

From Lemma \ref{l3.2} and the media value  we conclude the following Lemma \ref{l2.8x}.
\begin{lem}\label{l2.8x}
 Let $\alpha\in (-n,0)$ and $\phi $  be a Young function satisfying \eqref{delta0} and \eqref{deltafz}.

(i) If  $\Omega\subset\rn$ is a bounded  globally $n$-regular domain,
 then there exists a constant $C>0$ depending only on $n, \alpha,\phi$ and $\theta$ such that
 $$ \|u-u_\Omega\|_{L^{n/|\alpha|}(\Omega)}\le C\|u\|_{{\bf \dot B}^{\alpha,\phi}(\Omega)},\quad
 \mbox{ $\forall\ u \in {\bf\dot B}^{\alpha,\phi}(\Omega)$.}
$$

(ii)  If  $\Omega\subset\rn$ is a unbounded   globally $n$-regular domain,
 then  there exists a constant $C>0$ depending only on $n, \alpha,\phi$ and $\theta$ such that $$ \|u \|_{L^{n/|\alpha|}(\Omega)}\le C\|u\|_{{\bf\dot B}^{\alpha,\phi}(\Omega)},\quad
 \mbox{ $\forall\ u \in {\bf \dot B}^{\alpha,\phi}(\Omega)$  with
 $|\{x\in\Omega: |u|>a\}|<\fz$ for all $a>0$.}
$$

% (iii) If $\Omega=\rn$, then there exists a constant $C>0$ depending only on $n, \alpha,\phi$  such that
%  $$ \|u-u_B\|_{L^{n/|\alpha|} (B )}\le C\|u\|_{{\bf\dot B}^{\alpha,\phi} (B )}\quad\forall {\ \rm balls }\ B\ {\rm and}\  u\in {\bf\dot{B}}^{\alpha,\phi} (B) $$
% and $$\inf_{c\in\rr}\|u-c\|_{L^{n/|\alpha|}(\rn)}\le C\|u\|_{{\bf \dot B}^{\alpha,\phi}(\rn)},\quad
% \mbox{ $\forall\ u \in {\bf\dot B}^{\alpha,\phi}(\rn)$.}$$
   \end{lem}

\begin{proof}
  (i) Suppose that $\Omega \subset\rn$ is a bounded globally $n$-regular domain.
 For any  $u \in {\bf\dot B}^{\alpha,\phi}(\Omega)$,
set  the median value $$m_u(\Omega):=\inf\left\{c\in\rr: |\{x\in B: u>c \}|\le \frac12|\Omega|\right\}.$$
Then
$$|\{x\in \Omega: u>m_u(\Omega) \}|\le \frac12|\Omega|\quad{\rm and }\quad|\{x\in \Omega: u<m_u(\Omega) \}|\le \frac12|\Omega|.$$
Write $u_+=[u-m_u(\Omega)]\chi_{u\ge m_u(\Omega)}$ and $u_-=-[u-m_u(\Omega)]\chi_{u\le m_u(\Omega)}$.
Then  $u_{\pm}$ satisfies \eqref{condu}, and hence by Lemma \ref{l3.2}, we obtain
$\|u_\pm\| _{L^{n/|\alpha|}(\Omega)} \le C \|u_\pm\|_{{\bf \dot B}^{\alpha,\phi}(\Omega)}$,

On the other hand, note that  $u-m_u(\Omega)=u_+-u_-$ and
$$\|u-m_u(\Omega)\|_{L^{n/|\alpha|}(\Omega)}^{n/|\alpha|}= \|u_+\|_{L^{n/|\alpha|}(\Omega)}^{n/|\alpha|}+
 \|u_-\|_{L^{n/|\alpha|}(\Omega)}^{n/|\alpha|}.$$
 Moreover, by
 $$|u(x)-u(y)|=|[u(x)-m_u(\Omega)]-[u(y)-m_u(\Omega)]|= |u_+(x)-u_+(y)|+|u_-(x)-u_-(y)|,\quad\forall x,y\in\Omega,$$
 we have $\|u_\pm\|_{{\bf\dot B}^{\alpha,\phi}(\Omega)}\le \|u\|_{{\bf \dot B}^{\alpha,\phi}(\Omega)}$.
 We then conclude
  $$\|u- u_\Omega \|_{L^{n/|\alpha|}(\Omega)}\le 2\|u-m_u(\Omega)\|_{L^{n/|\alpha|}(\Omega)}\le C\|u\|_{{\bf \dot B}^{\alpha,\phi}(\Omega)}$$
 as desired.

(ii) Suppose that $\Omega \subset\rn$ is a unbounded globally $n$-regular domain.
Assume that $u \in {\bf \dot B}^{\alpha,\phi}(\Omega)$ satisfies  $|\{x\in\Omega: |u|>a\}|<\fz$ for any $a>0$.
Write $u_+=u\chi_{\{x\in\Omega: u(x)\ge0\}}$ and
$u_-=-u\chi_{\{x\in\Omega: u(x)\le0\}}$. Then $u=u_+-u_-$,
and  $u_\pm$ satisfies \eqref{condufz}. By Lemma \ref{l3.2},
$\|u_\pm\| _{L^{n/|\alpha|}(\Omega)} \le C \|u_\pm\|_{{\bf\dot B}^{\alpha,\phi}(\Omega)}$.
Note that
 $$|u(x)-u(y)|= |u_+(x)-u_+(y)|+|u_-(x)-u_-(y)|,\quad\forall x,y\in\Omega,$$
and hence, $\|u_\pm\|_{{\bf\dot B}^{\alpha,\phi}(\Omega)}\le \|u\|_{{\bf\dot B}^{\alpha,\phi}(\Omega)}$.
Combining with $\|u \|_{L^{n/|\alpha|}(\Omega)}^{n/|\alpha|}= \|u_+\|_{L^{n/|\alpha|}(\Omega)}^{n/|\alpha|}+
 \|u_-\|_{L^{n/|\alpha|}(\Omega)}^{n/|\alpha|},$
 we conclude $\|u \| _{L^{n/|\alpha|}(\Omega)}
 \le C \|u \|_{{\bf\dot B}^{\alpha,\phi}(\Omega)}$ as desired. This completes the proof of Lemma \ref{l2.8x}.
  \end{proof}

 Theorem \ref{mainthm0} then follows from Lemma \ref{l2.8x}.

\begin{proof}[Proof of Theorem 1.1.]
Note that $\{B(z,R)\}_{z\in\rn,R>0}$  are globally $n$-regular domains
  with the same constant $\theta$.  Indeed, let $\theta>0$ such that
 $$ \mbox{$|B(0,1)\cap B(x,r )|\ge \theta r^n$ for all $x\in B(0,1)$ and $r<2$}.$$
   Then for any $z\in\rn$ and $R>0$, we have
 $$ |B(z,R)\cap B(x,r)|=|B(0,R)\cap B(x-z,r)|=R^n|B(0,1)\cap B((x-z)/R,r/R)|\ge \theta R^n (r/R)^n= \theta r^n\quad $$
 whenever $0<r<2R$ and $x\in B(z,R)$.
  Thus by Lemma \ref{l2.8x} (i)  we know that
  there exists a constant $C>0$ such that
    $$ \|u-u_B\|_{L^{n/|\alpha|} (B )}\le C\|u\|_{{\bf\dot B}^{\alpha,\phi} (B )}\quad\forall {\ \rm balls }\ B\ {\rm and}\  u\in {\bf\dot{B}}^{\alpha,\phi} (B) $$
    as desired.
Especially, given any $u\in{\bf\dot{B}}^{\alpha,\phi}(\rn)$    we have    $$ \|u-u_{B(0,2^k)}\|_{L^{n/|\alpha|} (B(0,2^k))}\le C\|u\|_{{\bf\dot B}^{\alpha,\phi} (B(0,2^k) )} \le  C\|u\|_{{\bf\dot B}^{\alpha,\phi} (\rn)}\quad\forall k\in\nn.$$
 Therefore
  $$|u_{B(0,2^{k})}-u_{B(0,2^{k+1})} |\le \omega_n^{\alpha/n}2^{-k|\alpha|}\|u -u_{B(0,2^{k+1})} \|
  _{L^{n/|\alpha|} (B(0,2^{k+1}))}\le C2^{-k|\alpha|}
  \|u\|_{{\bf\dot B}^{\alpha,\phi} (\rn)}\quad\forall k\in\nn.$$
 Thus $u_{B(0,2^{k})}$ converges to some $c\in\rr$ as $k\to\fz$, and $$|u_{B(0,2^{k})}-c|\le \sum\limits_{l\ge k}|u_{B(0,2^{l})}-u_{B(0,2^{l+1})}|\le \sum\limits_{l\ge k}  C2^{-l|\alpha|}
  \|u\|_{{\bf\dot B}^{\alpha,\phi} (\rn)}\le C\frac{2^{-k|\alpha|}}{1-2^\alpha}
  \|u\|_{{\bf\dot B}^{\alpha,\phi} (\rn)}.$$
   Since
\begin{align*} \|u -c  \|
  _{L^{n/|\alpha|} (B(0,2^{k }))}
 & \le
 \|u -u_{B(0,2^{k})} \|
  _{L^{n/|\alpha|} (B(0,2^{k }))}+ |B(0,2^{k })|^{|\alpha|/n}|u_{B(0,2^{k})}-c|
 \le C\|u\|_{{\bf\dot B}^{\alpha,\phi} (\rn)},
 \end{align*}
 letting $k\to\fz$, we obtain
  $ \|u-c\|_{L^{n/|\alpha|} (\rn)}\le C\|u\|_{{\bf\dot B}^{\alpha,\phi} (\rn)} $
  as desired. This completes the proof of Theorem 1.1.
\end{proof}

%  The proof  in Step 1 is motivated by \cite{xx,cv,sv}, while the proof in Step 2 is However, the proof in Step 1 gives  something more than \cite{xx,cv,sv}.

\begin{rem}\rm
(i) Let $0<s<1$ and
 $1\le p<n/s$. It is proved by \cite[Corollary 25]{cv} and   \cite[Lemma A.1]{sv13} that
 \begin{equation}\label{e3.x8}
\int_{ \rn\setminus E  } {|x-y|^{n+sp}}dy \ge   \frac{C}{|E|^{sp}}
\quad\forall \, \mbox{$E\subset\rn$   with $0<|E|<\fz$},
\end{equation}
 that is, Lemma \ref{l2.7x}  with $\Omega=\rn$, $\alpha=s-n/p$ and $\phi(t)=t^p$.
 Using \eqref{e3.x8}, Di Nezza et al \cite{dpv} proved that
 \begin{eqnarray*}
&&\int_\rn\int_\rn \frac{|u(x)-u(y)|^p}{ |x-y|^{n+sp}}\,d
 {dxdy} \ge   \sum_{\substack{l\in\zz, a_{l-1}\neq0}}
 a_l  a_{l-1} ^{-sp}   \quad\forall   u\in  {\bf\dot B}^s_{pp}(\rn) \mbox{ having bounded supports},
\end{eqnarray*}
which is Lemma \ref{l3.2} with $\Omega=\rn$, $\alpha=s-n/p$ and $\phi(t)=t^p$ essentially.
After several technical arguments, this allows them to obtain
 \begin{equation}\label{e3.x9}
 \|u\|_{L^{np/(n-sp)}(\rn)}\le \|u\|_{{\bf\dot B}^s_{pp}(\rn)} \quad
 \mbox{  $\forall \, u\in  {\bf\dot B}^s_{pp}(\rn)$ having bounded supports,}
 \end{equation}
 See \cite[Section 6]{dpv} for details.    To get \eqref{e3.x9},
 our  proof in Lemma \ref{l3.2} via Orlicz norm   simplify the argument in \cite[Section 6]{dpv}
 by dropping several technical arguments therein.

  (ii)  Lemma \ref{l2.7x}  extends  \eqref{e3.x8} not only  to general $\phi$ but also to globally $n$-regular  domains  $\Omega$.
 Applying Lemma \ref{l2.7x} and an argument
 simpler than \cite[Section 6]{dpv},
 we extend \eqref{e3.x9} to not only  general $\phi$ but also to globally $n$-regular  domains  $\Omega$ as in Lemma \ref{l3.2}.
  Moreover, with the aid of   the median value, we further
  obtain the desired imbedding  of  ${\bf\dot B}^{\alpha,\phi} (\Omega)$  in Lemma 3.2.
 % The imbedding  of  ${\bf\dot B}^{\alpha,\phi} (B)$ for all balls $B\subset\rn$  allows to obtain
%  Corollary \ref{c3.1}.
  In particular, %applying Theorem \ref{mainthm} (i) to balls $B\subset\rn$,
%  $\alpha=s-n/p$ and $\phi(t)=t^p$, and
%   applying  Corollary \ref{c3.1} to
%  $\alpha=s-n/p$ and $\phi(t)=t^p$, we shows that
   %the modification in this paper of the idea by xx
   we give a new and direct proof   to the well-known facts
   \eqref{ebimb} and \eqref{ebimrn}, and also \eqref{ps41} for bounded domains.
   %
%$$
%   \inf_{c\in\rr} \|u-c\|_{L^{np/(n-sp)}(B)}\le C\|u\|_{{\bf\dot B}^s_{pp}(B)}\quad
%   \forall \mbox{balls $B$ and } u\in  {\bf\dot B}^s_{pp}(B)$$
%  and
%  $$
%\inf_{c\in\rr} \|u-c\|_{L^{np/(n-sp)}(\rn)}\le C\|u\|_{{\bf\dot B}^s_{pp}(\rn)}
%\quad u\in  {\bf\dot B}^s_{pp}(\rn).
%$$

(iii) When $\Omega\subset\rn$ is a globally $n$-regular domain and $|\Omega|=\fz$,
%\begin{equation}\label{e3.x10}\inf_{c\in\rr}\|u-c\|_{L^{n/|\alpha|} (\Omega)}\le \|u\|_{{\bf\dot B}^{\alpha,\phi} (\Omega)}\quad\forall u\in  {\bf\dot B}^{\alpha,\phi} (\Omega),
%\end{equation}
  the direct argument above fails to prove  \eqref{bubdah}; the difficulty is to find
a sequence domains $\Omega_R$ which are globally $n$-regular  with
  the same $\theta$ so that $\Omega _R$ is increasing and converges to $\Omega $. % Instead, in the forth coming paper, we will show that
%$\Omega$ is ${\bf\dot B}^{\alpha,\phi}$-extension domains, then derive \eqref{bubdah} from Lemma \ref{l2.8x} (iii).
\end{rem}

Now we prove Theorem \ref{mainthm} and Corollary 1.3.

  \begin{proof}[Proof of Theorem \ref{mainthm}.]

If $\Omega\subset\rn$ is a globally $n$-regular domain, then \eqref{bdah} and \eqref{ubdah} follows from Lemma \ref{l2.8x} directly.
Below assume that  $\Omega\subset\rn$ is a domain
satisfying \eqref{bdah} or \eqref{ubdah}.
With the aid of Lemma \ref{l2.3}, \eqref{condu} and \eqref{condufz},
 and by borrowing some ideas from \cite{hkt08,z14}, we will show that
$\Omega$ is globally $n$-regular.
To this end, take arbitrary $z\in \Omega$ and $0<r< \frac12\diam\Omega$,
and for $j\ge0$, let  $0<b_j\le 1$   such that
\begin{eqnarray}\label{e3.1}
|B(z,b_jr)\cap\Omega| =\frac{1}{2^j}|B(z, r)\cap\Omega|.
\end{eqnarray}
Obviously, $1=b_0>b_j>b_{j+1}>0$ for all $j\ge1$
and
\begin{eqnarray}\label{e3.1x}
|B(z,b_jr)\cap\Omega|=\frac{1}{2}|B(z,b_{j-1}r)\cap\Omega|\quad\forall j\ge0.
\end{eqnarray}

{\it  Case $\diam\Omega=\fz$.}
It suffices to prove that there exists a constant $C>0$ independent of $z, r$ such that
\begin{equation}\label{e3.x2}  \phi \left(  C \frac{|B_{\Omega}(z, b_{j }r)|^{\alpha/n}}{(b_jr-b_{j+1}r)^{\alpha}} \right) \left[  \frac{(b_jr-b_{j+1}r)^n}{|B_\Omega(z,b_jr)| }\right]^{-1}\ge 1\quad\forall j\ge0.
\end{equation}
Indeed, since $\phi( Cs^{-\alpha}) s^{-n}\to 0$ as $s\to \fz$ as given in Lemma 2.1 (ii), we know that
$\phi( Cs^{-\alpha}) s^{-n}\ge 1$
implies that $s\le \Lambda_C $ for some constant $\Lambda_C>0$.
By this and \eqref{e3.x2}, we obtain
$$\frac{(b_jr-b_{j+1}r) }{|B_{\Omega}(z, b_{j }r)|^{ 1/n}}\le \Lambda_C.$$
This together with \eqref{e3.1} yields that
$$ b_jr -b_{j+1}r\le \Lambda_C |B_{\Omega}(z, b_{j }r)|^{1/n}
= \Lambda_C 2^{-j/n}|B_{\Omega}(z,  r)|^{ 1/n},$$
which gives
$$r =\sum_{j\ge0}(b_jr -b_{j+1}r)\le \sum_{j\ge0} \Lambda_C 2^{-j/n}|B_{\Omega}(z,  r)|^{ 1/n}=\Lambda_C  |B_{\Omega}(z,  r)|^{ 1/n}$$
 as desired.

To prove \eqref{e3.x2},   for $j\ge0$  let $u_{z,b_{j+1}r, b_jr}$ be the
function defined by \eqref{defu}.
By Lemma \ref{l2.3}, we have
 $u_{z,b_{j+1}r, b_jr}\in {\bf\dot B} ^{\alpha,\phi}(\Omega)$ and
$$\|u_{z,b_{j+1}r, b_jr} \|_{{\bf\dot B} ^{\alpha,\phi}(\Omega)}\le C (b_jr-b_{j+1}r)^{-\alpha}\left [\phi^{-1}\left(\frac{(b_jr-b_{j+1}r)^n}{|B_\Omega(z,b_jr)| }\right)\right]^{-1}.$$

Since $|\{x\in\Omega: u_{z,b_{j+1}r, b_jr}(x)\ne 0\}|<\fz$,
by \eqref{ubdah} we have $$\|u_{z,b_{j+1}r, b_jr} \|_{L^{n/|\alpha|}(\Omega)}\le C
\|u_{z,b_{j+1}r, b_jr} \|_{{\bf\dot B} ^{\alpha,\phi}(\Omega)}.$$
Note that
$$\|u_{z,b_{j+1}r, b_jr} \|_{L^{n/|\alpha|}(\Omega)}\ge  |\Omega\cap B(x, b_{j+1}r)|^{-\alpha/n}=2^{\alpha/n}|\Omega\cap B(x, b_{j }r)|^{-\alpha/n}.$$
We conclude that
$$2^{\alpha/n}|B_{\Omega}(z, b_{j }r)|^{-\alpha/n}\le C (b_jr-b_{j+1}r)^{-\alpha}\left [\phi^{-1}\left(\frac{(b_jr-b_{j+1}r)^n}{|B_\Omega(z,b_jr)| }\right)\right]^{-1},$$
which implies that
$$\phi^{-1}\left(\frac{(b_jr-b_{j+1}r)^n}{|B_\Omega(z,b_jr)| }\right)
\le    C \frac{|B_{\Omega}(z, b_{j }r)|^{\alpha/n}}{(b_jr-b_{j+1}r)^{\alpha}}.$$

{\it Case $ \diam\Omega<\fz$.}  Note that by a simialr argument as in the case $|\Omega|=\fz$, we have
$b_1r<\Lambda_C|B_\Omega(z,b_1r)|^{1/n}.$
Indeed,  for $j\ge1$,
\begin{equation}\label{e3.x3} \|u_{z,b_{j+1}r, b_jr}-(u_{z,b_{j+1}r, b_jr})_{\Omega}\|_{L^{n/|\alpha|}(\Omega)}\ge \frac12 |\Omega\cap B(x, b_{j+1}r)|^{-\alpha/n}=\Big(\frac12\Big)^{1-\alpha/n}|\Omega\cap B(x, b_{j }r)|^{-\alpha/n}.
\end{equation}
Note that $u_{z,b_{j+1}r, b_jr} -(u_{z,b_{j+1}r, b_jr})_{\Omega}\ge \frac12$ either in $B_\Omega (x,b_{j+1}r)$
or in $\Omega\setminus B_\Omega (x,b_{j}r)$. Since $j\ge 1 $  implies that
$B_\Omega (x,b_{j-1}r)\setminus B_\Omega (x,b_{j }r)\subset \Omega\setminus B_\Omega (x,b_{j}r)$
and hence
$$|\Omega  \setminus B_\Omega (x,b_{j }r)|\ge
|B_\Omega (x,b_{j-1}r)\setminus B_\Omega (x,b_{j }r)|
=  | B_\Omega (x,b_{j }r) |>| B_\Omega (x,b_{j+1}r) |, $$
together with
$|B_\Omega (x,b_{j+1}r)|= \frac12| B_\Omega (x,b_{j }r) |$
we obtain \eqref{e3.x3}. Then by   the same argument as in the case $\diam(\Omega)=\fz$, we are able to prove that
$b_1r<\Lambda_C|B_\Omega(z,b_1r)|^{1/n},$ where the value $\Lambda_C$ needed to be adjusted.

If
$b_1 \ge 1/10  $, by $b_1r<\Lambda_C|B_\Omega(z,b_1r)|^{1/n}$,
we have   $ r<10 \Lambda_C|B_\Omega(z, r)|^{1/n}$ as desired.
Assume that $b_1<1/10$.
  Let $R=\frac25r$ and $y\in B_\Omega(z,r)$ with $|y-z|=b_1r+R/2$.
  Note that since $\Omega$ is  path-connected  and $b_1r+R/2<3\diam\Omega/20$, such $y $ exists. Then $B(z,b_1r)\subseteq B(y,R)\subsetneq B(z,r)$, and  $B(z,b_1r)\cap B(y,R/2)=\emptyset$. Thus if $|B_{\Omega}(y,\wz b_1R)|=\frac12|B_{\Omega}(y, R)|$,
  by $|B_{\Omega}(z,b_1r)|=\frac12|B_\Omega(z,r)|> \frac12|B_\Omega(y,R)|$,   we have $B_{\Omega}(y,\wz b_1R)\cap B_{\Omega}(z,b_1r)\neq\emptyset$, so $\wz b_1\ge 1/2$. Note that by the  above argument, we already have
    $\wz b_1R \le \Lambda_C|B_\Omega(y,\wz b_1R)|^{1/n}$.
    Hence $$\frac25r= R\le 2\Lambda_C|B_\Omega(y, R)|^{1/n}\le  2\Lambda_C|B_\Omega(z, r)|^{1/n},$$
    which gives $5\Lambda_C|B_\Omega(z, r)|^{1/n}\ge  r  $ as desired.
    \end{proof}

  \begin{proof}[Proof of Corollary \ref{c1.2}] (i) Assume that $\Omega\subset\rn$ is
  a bounded  $n$-regular domain. Then it is also  bounded globally $n$-regular domain. Let $u\in {\bf B}^{\alpha,\phi}(\Omega)$.
  Since Jessen's inequality implies that
   $$\phi\left(\frac{|u_\Omega|}\lambda\right)\le
  \phi\left(\fint_\Omega\frac{|u|}\lambda\,dx\right)\le   \fint_\Omega\phi\left(\frac{|u|}\lambda\right)\,dx\le \frac1{|\Omega|}\int_\Omega\phi\left(\frac{|u|}\lambda\right)\,dx\quad\forall \lambda>\|u\|_{L^\phi(\Omega)},$$
  we have $|u_\Omega|\le \|u\|_{L^\phi(\Omega)} \phi^{-1}(|\Omega|^{-1})$.
  By Lemma \ref{l2.8x} (i), we then have
  $$\|u\|_{L^{n/|\alpha|}(\Omega)}\le
  \|u-u_\Omega\|_{L^{n/|\alpha|}(\Omega)}+|u_\Omega||\Omega|^{ |\alpha|/n}\le
    C   \|u\|_{{\bf B}^{\alpha,\phi}(\Omega)}$$
  as desired.

Conversely, assume that $\Omega\subset\rn$ is
  a bounded  domain satisfying  that
    $$\|u\|_{L^{n/|\alpha|}(\Omega)}
  \le C   \|u\|_{{\bf B}^{\alpha,\phi}(\Omega)}\quad\forall u\in {\bf B}^{\alpha,\phi}(\Omega) $$
  for some constant $C>0$.  Considering Lemma \ref{l2.4},
    % Letting $u_{x,r,t}$ be as in \eqref{defu},by Lemma \ref{l2.3} we conclude that
%    $$\|u_{x,r,t}\|_{\dot {\bf B} ^{\alpha, \phi}(\Omega)}\le C (t-r)^{\alpha}\left [\phi^{-1}\left(\frac{(t-r)^n}{|B_\Omega(x,t)| }\right)\right]^{-1}.$$
 by the same argument as in the case $|\Omega|=\fz$ in Theorem 1.1, we
  show that $\Omega$ is  $n$-regular; the details are  omitted.
The proof of Corollary \ref{c1.2} is completed.
  \end{proof}

\medskip

\noindent {\bf Acknowledgements.} The   author  would like to thank Professor Yuan Zhou
 for  several valuable discussions.  The author also would like to
 thank the supports of
 National Natural Science of Foundation of China (No. 11601494).

\medskip

%\noindent
%Department of Mathematics, Beijing University of Aeronautics and Astronautics, Beijing 100083, P. R. China

%\noindent{\it E-mail address}:  %\texttt{wangzhuangzhao@gmail.com}


\begin{thebibliography}{12}

\bibitem{a75}
R. A. Adams and J. J. F.  Fournier,  \emph{Sobolev spaces},
Elsevier/Academic Press, Amsterdam, 2003.
%
%\bibitem{b82}
%J. Boman, \emph{$L^p$-estimates for every strongly elliptic systems}, umpublished manuscript.

%
%\bibitem{b88}
%B. Bojarski,
%\emph{Remarks on Sobolev imbedding inequalities},
%Complex analysis, Joensuu 1987,  52-68, Lecture Notes in Math., 1351, Springer, Berlin, 1988.

%
%\bibitem{bk95} S. M. Buckley and P. Koskela,
%\emph{Sobolev-Poincar\'e implies John},  Math. Res. Lett.  {\bf2} (1995),  577-594.

%
%\bibitem{bk96} S. M. Buckley and P. Koskela,
%\emph{Criteria for imbeddings of Sobolev-Poincar\'e type},
% Internat. Math. Res. Notices {\bf18} (1996), 881-901.

%
%\bibitem{bo99}
%S. M. Buckley and J. O'Shea,
%\emph{Weighted Trudinger-type inequalities}, Indiana Univ. Math. J.  {\bf48}  (1999),  85-114.

%
%
%\bibitem{cw}
%S.-K. Chua and R. L. Wheeden,
%\emph{Self-improving properties of inequalities of Poincar\'e type on measure spaces and applications},
%J. Funct. Anal.  {\bf255}  (2008),  2977-3007.
%
\bibitem{cv}
L. Caffarelli and E. Valdinoci, Uniform estimates and limiting arguments for nonlocal minimal surfaces, Calc. Var.
Partial Differential Equations 41   (2011) 203--240.
%
\bibitem{dpv}
E. Di Nezza,  G.  Palatucci and E. Valdinoci,   Hitchhiker's guide to the fractional Sobolev spaces. Bull. Sci. Math. 136 (2012),   521--573.
%\bibitem{div16}
%B. Dyda,  L. Ihnatsyeva,  A. V. V\"ah\"akangas,
% On improved fractional Sobolev-Poincar\'e inequalities. Ark. Mat. 54 (2016),   437-454.
% \bibitem{hv13}
% R. Hurri-Syrj\"anen, A.V. V\"ah\"akangas,  On fractional Poincar\'e inequalities. J. Anal. Math. 120 (2013), 85¨C104.

%\bibitem{g77}
%F. W. Gehring,
%\emph{Univalent functions and the Schwarzian derivative},
%Comment. Math. Helv. {\bf52} (1977), 561-572.



%\bibitem{go79} F. W. Gehring and B. G. Osgood,
%\emph{Uniform domains and the quasihyperbolic metric},
%J. Analyse Math. {\bf36} (1979), 50-74.


%\bibitem{gm85b}
%F. W. Gehring and O. Martio,
%\emph{Quasiextremal distance domains and extension of quasiconformal mappings},
%J. Analyse Math.  {\bf45} (1985), 181-206.

\bibitem{gt}
D. Gilbarg and N. S. Trudinger,
 Elliptic partial differential equations of second order,
Springer-Verlag, Berlin, 2001.
\bibitem{gkz13} A. Gogatishvili, P. Koskela and Y. Zhou,   Characterizations of Besov and Triebel-Lizorkin spaces on
metric measure spaces, Forum Math. 25 (2013), 787-819.








%\bibitem{g82} V. M. Gol'd\v{s}te\v{\i}n,
%\emph{Quasiconformal, quasi-isometric mappings and the Sobolev spaces},
%Complex analysis and applications '81 (Varna, 1981),  202-212,
%Publ. House Bulgar. Acad. Sci., Sofia, 1984.


%\bibitem{gr83} V. M. Gol'd\v{s}te\v{\i}n and Yu. G. Reshetnyak,
%\emph{Introduction to the theory of functions with generalized derivatives, and quasiconformal mappings (Russian)},
%Moscow, 1983. 285 pp.


%\bibitem{gv}
%V. M. Gol'd\v{s}te\v{\i}n and S. K. Vodop'anov,
%\emph{ Prolongement de fonctions diff\'erentiables hors de domaines plans},
%C. R. Acad. Sci. Paris S\'er. I Math. {\bf293} (1981), 581-584.


\bibitem{h96} P. Haj\l asz, {Sobolev spaces on an arbitrary metric spaces},
Potential Anal. {\bf5} (1996), 403-415.

\bibitem{h03} P. Haj\l asz,  {Sobolev spaces on metric-measure
spaces},  Heat kernels and analysis on manifolds, graphs,
and metric spaces (Paris, 2002), 173-218, Contemp.
Math., 338, Amer. Math. Soc., Providence, RI, 2004.


%\bibitem{hk00} P. Haj\l asz and P. Koskela,
%\emph{Sobolev met Poincar\'e}, Memoirs Amer. Math. Soc. {\bf145}(688) (2000),
%1-101.


\bibitem{hkt08}
P. Haj\l asz, P. Koskela and H. Tuominen,
 {Sobolev imbeddings, extensions and measure density condition},
 J. Funct. Anal.  { 254}  (2008),  1217-1234.


\bibitem{hkt08b}
P. Haj\l asz, P. Koskela and H. Tuominen,
 {Measure density and extendability of Sobolev functions},
Rev. Mat. Iberoam.  { 24} (2008), 645-669.


\bibitem{hit}
T. Heikkinen,  L. Ihnatsyeva and  H. Tuominen, Measure density and extension of Besov and Triebel-Lizorkin functions
J.  Four. Anal.  Appl. 22 (2016), 334-382.

%\bibitem{k90}
%P. Koskela, \emph{Capacity extension domains},
%Ann. Acad. Sci. Fenn. Ser. A I Math. Dissertationes  No. 73  (1990), 42 pp.


%\bibitem{ks08} P. Koskela and E. Saksman,
%\emph{Pointwise characterizations of Hardy-Sobolev functions},
%Math. Res. Lett.  {\bf15} (2008),  727-744.

%
%\bibitem{kyz09} P. Koskela, D. Yang and Y. Zhou,
%\emph{A characterization of Haj\l asz-Sobolev
%and Triebel-Lizorkin spaces}, J. Funct. Anal. {\bf258} (2010), 2637-2661.


%\bibitem{kyz09c} P. Koskela, D. Yang and Y. Zhou,
%\emph{Pointwise characterizations of Besov and Triebel-Lizorkin spaces and quasiconformal mappings},
%submitted.


%\bibitem{mp}
%P. MacManus and C. P\'erez,  \emph{Trudinger inequalities without derivatives},
%Trans. Amer. Math. Soc.  {\bf354}  (2002),  1997-2012.

\bibitem{jw78} A. Jonsson and  H. Wallin, A Whitney extension theorem in $L^p$ and Besov spaces,
 Ann. Inst. Fourier (Grenoble) 28 (1978),  139-192.
\bibitem{jw84} A. Jonsson and H. Wallin,  Function spaces on subsets of $\rr^n$,
Math. Rep. 2 (1984), no. 1, xiv+221 pp.

\bibitem{lz} T. Liang and Y. Zhou, Orlicz-Sobolev extension and Ahlfor $n$-regualr domains, 2018. to appear.

%\bibitem{ms79}
%O. Martio and J. Sarvas,
%\emph{Injectivity theorems in plane and space},
%Ann. Acad. Sci. Fenn. Ser. A I Math.  {\bf4}  (1979), 383-401.


%\bibitem{m90}
%A. Miyachi, \emph{Hardy-Sobolev spaces and maximal functions}, Jour. Math. Soc. Japan  {\bf42}  (1990), 73-90.


%\bibitem{m91} A. Miyachi,
%\emph{Extension theorems for real variable Hardy and Hardy-Sobolev spaces},
%Harmonic analysis (Sendai, 1990),  170--182, ICM-90 Satell. Conf. Proc., Springer, Tokyo, 1991.

%
%\bibitem{nv91} R. N\"akki and J. V\"ais\"al\"a, \emph{John disks},
%Exposition Math. {\bf9} (1991), 3-44.

\bibitem{p76}
J. Peetre,  New thoughts on Besov spaces, Duke University Mathematics Series, No. 1. Mathematics Department, Duke University, Durham, N.C., 1976. vi+305 pp.
\bibitem{c} M. C. Piaggio, Orlicz spaces and the large scale geometry of Heintze groups, Math. Ann. 368 (2017), 433-481.
%
%\bibitem{ss90} W. Smith and D. A. Stegenga,
%\emph{H\"older domains and Poincar\'e domains},
%Trans. Amer. Math. Soc. {\bf319} (1990), 67-100.


%\bibitem{ss90b} W. Smith and D. A. Stegenga,
%\emph{Sobolev imbeddings and integrability of harmonic functions on H\"older domains},
%Potential theory (Nagoya, 1990),  303-313, de Gruyter, Berlin, 1992.


%\bibitem{t83}
%H. Triebel,  \emph{Theory of function spaces}, Birkh\"auser Verlag, Basel, 1983.
\bibitem {sv11}O. Savin, E. Valdinoci, Density estimates for a nonlocal variational model via the Sobolev inequality, SIAM J.
Math. Anal. 43   (2011), 2675-2687.

\bibitem {sv13}O. Savin, E. Valdinoci, Density estimates for a variational model driven by the Gagliardo norm, J. Math. Pures   Appl. (101) 2014, 1-26.

\bibitem{s06}
P. Shvartsman, Local approximations and intrinsic
characterizations of spaces of smooth functions on regular subsets of $\rn$, Math. Nachr. 279 (2006), 1212-1241.



\bibitem{s07}
P. Shvartsman,  On extensions of Sobolev functions defined on rgular subsets of metric measure spaces, Journal of Approximation Theory.  214 (2007), 139-161.

%\bibitem{s10} P.
%Shvartsman,  \emph{On Sobolev extension domains in $R^n$}, J. Funct. Anal.  {\bf258} (2010), 2205-2245.

%
%\bibitem{s93} E. M. Stein, \emph{Harmonic analysis: real-variable
%methods, orthogonality, and oscillatory integrals}, Princeton Univ.
%Press, Princeton, N. J., 1993.



%\bibitem{vgl} S. K. Vodop'janov, V. M. Gol'd\v{s}te\v{\i}n and
%T. G. Latfullin,   \emph{A criterion for the extension of functions of the class
%$L_{2}^{1}$ from unbounded plane domains},
%Sibirsk. Mat. Zh.  {\bf20}  (1979), 416-419.


\bibitem {y03} D. Yang,  New characterizations
of Haj\l asz-Sobolev spaces on metric spaces,
Sci. China Ser. A { 46} (2003), 675-689.

%\bibitem{bo88} B. Bojarski,  Remarks on Sobolev imbedding inequalities. Complex Analysis, Joensuu (1988), 52-68. Leture Notes in Math., 1351, Springer, Berlin, (1988).

%\bibitem{bk95} S.M. Buckley and P. Koskela, Sobolev-Poincar\'e implies John. Res.  Lett. 2 (1995), 577-594.

%\bibitem {bsk96}  S.M. Buckley and P. Koskela, Criterria for imbeddings of Sobolev-Poincar\'e type, Internat. Math. Res. Notices 18 (1996), 881-902.
%
%\bibitem{fgw94} B. Franchi, C. E. Guti$\rm{\acute{e}}$rrez, R. L. Wheeden, Weighted Sobolev¨CPoincar$\rm{\acute{e}}$ inequalities
%for Gru%sin type operators, Comm. P. D. E. 19 (1994), 523-604.
%
%\bibitem{flw95} B. Franchi, G. Lu, R. L. Wheeden, Representation formulas and weighted Poincar$\rm{\acute{e}}$
%inequalities for H$\rm{\ddot{o}}$rmander vector fields, Ann. Inst. Fourier 45 (1995), 577-604.




%\bibitem{fpw98}  B. Franchi, C. P$\rm{\acute{e}}$rez,  R. L. Wheeden,   Self-Improving Properties of John-Nirenberg and Poincar$\rm{\acute{e}}$
%Inequalities on Spaces of Homogeneous Type.   Journal of functional analysis (1) 153 (1998), 108-146.


%\bibitem{h96} P. Haj\l asz,  Sobolev spaces on an arbitrary metric space, Potential Anal. 5 (1996),
%403-415.

%\bibitem{h03}
%P. Haj\l asz,  Sobolev spaces on metric-measure spaces. In Heat kernels and Analysis on Manifolds,
%Graphs, and Metric Spaces (Paris, 2002), Contemporary Mathematics vol. 338, pp. 173-218. American
%Mathematical Society, Providence, RI (2003).

%\bibitem{hkt08}
%P. Haj\l asz, P. Koskela and H. Tuominen,
%Sobolev imbeddings, extensions and measure density condition,
% J. Funct. Anal.  254  (2008),  1217-1234.


%\bibitem{hkt08b}
%P. Haj\l asz, P. Koskela and H. Tuominen,
%Measure density and extendability of Sobolev functions,
%Rev. Mat. Iberoam.  24 (2008), 645-669.





%\bibitem{k93}
%P. Koskela and F. Reitich, Holder contunity of Sobolev functions and quasiconformal mappings, Math. Z. 213 (1993), 4557-472

%\bibitem{k98}
%P. Koskela, Extensions and imbeddings,  J. Funct. Anal.  159
%(1998), 369-384.

%\bibitem{ms79} O. Martio and J. Sarvas, Injectivity theorems in plane and space,
%Ann. Acad. Sci. Fenn. Ser. A I Math. 4 (1979),  383-401.
%
%\bibitem{nv91} R. N$\rm{\ddot{a}}$kki and J. V$\rm{\ddot{a}}$is$\rm{\ddot{a}}$l$\rm{\ddot{a}}$, John disks, Expo. Math . 9 (1991),3-44.

 \bibitem {sun2} H. Sun, Orlicz-Besov extension and imbedding,  preprint.
%
%\bibitem{z13} Y. Zhou, Criteria for Optimal Global Integrablity of Haj\l asz-Sobolev Functions, Illinois Journal of Mathematics. 55 (2011), 1083-1103.

\bibitem {z14} Y. Zhou, Fractional Sobolev extension and imbedding, Trans. Amer. Math. Soc. 367 (2015), 959-979.



%\bibitem{L} L. Grafakos, Classical and Modern Fourior Analysis, Pearson Education Inc. 2008.


%\bibitem{a55}
%N. Aronszajn, Boundary values of functions with finite Dirichlet integral, Techn. Report of
%Univ. of Kansas 14 (1955), 77-94.



%\bibitem[8]{ds93} R. A. DeVore, R. C. Sharpley,
% Besov spaces on domains in $R^d$,
%Trans. Amer. Math. Soc. 335 (1993), 843-864.




%
%\bibitem{g58}
%E. Gagliardo,  Propriet\`a di alcune classi di funzioni in pi\`u variabili,   Ricerche Mat.  7  (1958), 102-137.

%
%\bibitem{gm85} F. W. Gehring and O. Martio,
%Lipschitz classes and quasiconformal mappings,
% Ann. Acad. Sci. Fenn. Ser. A I Math. 10 (1985), 203-219.

%
%\bibitem{gm85b}
%F. W. Gehring and O. Martio, Quasiextremal distance domains and extension of quasiconformal mappings,
%J. Analyse Math.  45 (1985), 181-206.

%
%\bibitem{gt}
%D. Gilbarg and N. S. Trudinger,
%Elliptic Partial Differential Equations of Second Order,
%Springer-Verlag, Berlin, 2001.





%
%\bibitem{gv80}
%V. M. Gol'd\v{s}te\v{\i}n and S. K. Vodop'anov,
%Prolongement des fonctions de classe ${\mathcal L}^{1}_{p}$ et applications quasi conformes,
%C. R. Acad. Sci. Paris S\'er. A-B  290  (1980), A453-A456.

%
%\bibitem{gv}
%V. M. Gol'd\v{s}te\v{\i}n and S. K. Vodop'anov,
%Prolongement de fonctions diff\'erentiables hors de domaines plans,
%C. R. Acad. Sci. Paris S\'er. I Math. 293 (1981), 581-584.








%
%\bibitem{j81}
%P. W. Jones, Quasiconformal mappings and extendability of functions in Sobolev spaces,
% Acta Math.  147  (1981), 71-88.


%\bibitem{k98}
%P. Koskela, Extensions and imbeddings,  J. Funct. Anal.  159
%(1998), 369-384.









%
%\bibitem{r99}
%V. S. Rychkov, On restrictions and extensions of the Besov and Triebel-Lizorkin spaces
%with respect to Lipschitz domains,  J. London Math. Soc. (2)  60  (1999),  237-257.

%
%\bibitem{r00}  V. S. Rychkov, Linear extension operators for
%restrictions of function spaces to irregular open sets, Studia Math. 140 (2000), 141-162.

%
%\bibitem{s06}
%P. Shvartsman, Local approximations and intrinsic
%characterizations of spaces of smooth functions on regular subsets of $\rn$, Math. Nachr. 279 (2006), 1212-1241.

%
%\bibitem{s07}
%P. Shvartsman, On extension of Sobolev functions defined on regular subsets of metric measure spaces,
%J. Approx. Theory 144 (2007), 139-161.

%
%\bibitem{s10} P.
%Shvartsman,  On Sobolev extension domains in $R^n$, J. Funct. Anal.  258 (2010), 2205-2245.


%\bibitem{s58}
%L. N. Slobodeckij,  Generalized Sobolev spaces and their application to boundary problems for partial differential equations,  Leningrad. Gos. Ped. Inst. U\v cen. Zap.  197  (1958), 54-112.

%
%\bibitem{s70}
%E. M. Stein,  Singular integrals and differentiability properties of functions,
%Princeton Mathematical Series, No. 30, Princeton University Press, Princeton, N.J. 1970

%
%\bibitem{t83}
%H. Triebel,  Theory of Function Spaces. Monographs in Mathematics, 78. Birkh\"auser Verlag, Basel, 1983.

%
%\bibitem{t02}
%H. Triebel, Function spaces in Lipschitz domains and on Lipschitz manifolds.
%Characteristic functions as pointwise multipliers,
%Rev. Mat. Complut. 15 (2002), no. 2, 475-524.

%
%\bibitem{t08}
%H. Triebel,  Function spaces and wavelets on domains, EMS Tracts in Mathematics, 7. European Mathematical Society (EMS), Z\"orich, 2008. x+256 pp.

%
%\bibitem{vgl} S. K. Vodop'janov, V. M. Gol'd\v{s}te\v{\i}n and
%T. G. Latfullin,  A criterion for the extension of functions of the class
%$L_{2}^{1}$ from unbounded plane domains,
%Sibirsk. Mat. Zh.  20  (1979), 416-419.

%
%\bibitem {z11} Y. Zhou, Haj\l asz-Sobolev extension and imbedding, J. Math. Anal. Appl. 382 (2011) 577-593.


\end{thebibliography}
\end{document}